\documentclass[11pt]{amsart}
\usepackage[top=1.5in, bottom=1.5in, left=1.5in, right=1.5in]{geometry}
\expandafter\let\csname ver@amsthm.sty\endcsname\relax

\let\qedhere\relax

\usepackage[all]{xy}
\usepackage{graphicx} 
\usepackage{tikz}
\usetikzlibrary{patterns}
\usetikzlibrary{decorations.pathreplacing}
\usetikzlibrary{arrows}
\usetikzlibrary{decorations.markings}

\usepackage{amsmath}
\usepackage{amssymb}
\usepackage{enumerate}
\usepackage{mathdots}
\usepackage{mathtools}
\usepackage{yhmath}

\usepackage{hyperref}
\usepackage{amsthm}
\usepackage[capitalize,noabbrev]{cleveref}

\allowdisplaybreaks

\numberwithin{equation}{section}

\newtheorem{thm}{Theorem}[section]
\newtheorem{lemma}[thm]{Lemma}
\newtheorem{cor}[thm]{Corollary}
\newtheorem{prop}[thm]{Proposition}
\newtheorem{conj}[thm]{Conjecture}

\newtheorem{problem}[thm]{Problem}

\newtheorem{Definition}[thm]{Definition}

\newtheorem{Example}[thm]{Example}
\newenvironment{example}
  {\begin{Example}\rm}{\hfill$\lozenge$\end{Example}}

\newtheorem{Remark}[thm]{Remark}
\newenvironment{remark}
  {\begin{Remark}\rm}{\hfill$\lozenge$\end{Remark}}
  
\crefname{thm}{Theorem}{Theorems}
\crefname{lemma}{Lemma}{Lemmas}
\crefname{cor}{Corollary}{Corollaries}
\crefname{prop}{Proposition}{Propositions}
\crefname{conj}{Conjecture}{Conjectures}
\crefname{question}{Question}{Questions}
\crefname{problem}{Problem}{Problems}

\crefname{definition}{Definition}{Definitions}
\crefname{example}{Example}{Examples}
\crefname{remark}{Remark}{Remarks}

\DeclareMathOperator{\row}{Row}
\DeclareMathOperator{\Rvac}{Rvac}
\DeclareMathOperator{\supp}{supp}
\DeclareMathOperator{\Flip}{Flip}
\DeclareMathOperator{\Krew}{Krew}

\DeclareMathOperator{\NC}{NC}
\DeclareMathOperator{\NN}{NN}

\DeclareMathOperator{\rk}{rk}
\DeclareMathOperator{\height}{ht}

\DeclareMathOperator{\Cat}{Cat}
\DeclareMathOperator{\Nar}{Nar}

\DeclareMathOperator{\GL}{GL}
\DeclareMathOperator{\ad}{ad}

\newcommand{\emailhref}[1]{\email{\href{#1}{#1}}}

\newcommand{\dfn}[1]{\textcolor{blue}{\emph{#1}}}

\title[Symmetry of Narayana numbers and rowvacuation of root posets]{Symmetry of Narayana numbers and rowvacuation of root posets}

\author[C. Defant]{Colin Defant}\emailhref{cdefant@princeton.edu}
\address{Department of Mathematics, Princeton University, Princeton, NJ 08544}
\thanks{C.D. was supported by a Fannie and John Hertz Foundation Fellowship and an NSF Graduate Research Fellowship. }

\author[S. Hopkins]{Sam Hopkins}\emailhref{shopkins@umn.edu}
\address{School of Mathematics, University of Minnesota, Minneapolis, MN 55455}
\thanks{S.H. was supported by an NSF Mathematical Sciences Postdoctoral Research Fellowship.}

\keywords{Coxeter--Catalan combinatorics, dynamical algebraic combinatorics, root posets, nonnesting partitions, noncrossing partitions, rowmotion, rowvacuation}

\date{\today}

\begin{document}

\begin{abstract}
For a Weyl group $W$ of rank $r$, the $W$-Catalan number is the number of antichains of the poset of positive roots, and the $W$-Narayana numbers refine the $W$-Catalan number by keeping track of the cardinalities of these antichains. The $W$-Narayana numbers are symmetric, i.e., the number of antichains of cardinality $k$ is the same as the number of cardinality $r-k$. However, this symmetry is far from obvious. Panyushev posed the problem of defining an involution on root poset antichains that exhibits the symmetry of the $W$-Narayana numbers. 

Rowmotion and rowvacuation are two related operators, defined as compositions of toggles, that give a dihedral action on the set of antichains of any ranked poset. Rowmotion acting on root posets has been the subject of a significant amount of research in the recent past. We prove that for the root posets of classical types, rowvacuation is Panyushev's desired involution.
\end{abstract}

\maketitle

\section{Introduction and statements of results} \label{sec:intro}

Let $\Phi$ be a crystallographic \dfn{root system} in an $r$-dimensional Euclidean space $V$, and $W\subseteq \GL(V)$ its corresponding \dfn{Weyl group}. \dfn{Coxeter--Catalan combinatorics} is the study of the \dfn{$W$-Catalan number}
\begin{equation} \label{eq:cat_prod}
 \Cat(W) \coloneqq \prod_{i=1}^{r} \frac{d_i+h}{d_i},
\end{equation}
where $d_1\leq \cdots\leq d_r$ are the \dfn{degrees} of $W$, and $h=d_r$ is its \dfn{Coxeter number}. Although not obvious, it is true that $\Cat(W)$ is an integer. It counts several collections of objects associated to~$\Phi$, including:
\begin{itemize}
\item \dfn{$W$-nonnesting partitions}, i.e., antichains of the poset $\Phi^+$ of positive roots;
\item \dfn{$W$-noncrossing partitions}, i.e., elements of $W$ between the identity $e$ and a fixed Coxeter element $c$ in absolute order.
\end{itemize}
We let $\NN(W)$ and $\NC(W)$ denote the sets of $W$-nonnesting partitions and $W$-noncrossing partitions, respectively. See, e.g.,~\cite[Chapter~1]{armstrong2009generalized} for a general introduction to Coxeter--Catalan combinatorics.

Though there has been a tremendous amount of work done in Coxeter--Catalan combinatorics in the past 20 plus years, the connection between the \dfn{nonnesting} and \dfn{noncrossing worlds} remains deeply mysterious. As an example of the divide between nonnesting and noncrossing, we note that there is a uniform proof of the product formula~\eqref{eq:cat_prod} for the nonnesting objects~\cite{ cellini2002adnilpotent, haiman1994conjectures}, but the only known proof of this formula for the noncrossing objects is case-by-case~\cite{bessis2003dual}. On the other hand, the noncrossing definition of the $W$-Catalan number extends directly to all finite Coxeter groups, while the naive generalization of the nonnesting definition fails to work properly in the non-crystallographic types.

Our present work focuses on another difference between the nonnesting and noncrossing worlds, this one concerning a refinement of the $W$-Catalan number. Namely, for $0\leq k \leq r$, the $k$th \dfn{$W$-Narayana number} $\Nar(W,k)$ can be defined as either:
\begin{itemize}
\item the number of nonnesting partitions in $\NN(W)$ of cardinality $k$;
\item the number of noncrossing partitions in $\NC(W)$ of rank $k$.
\end{itemize}
Evidently, the Narayana numbers refine the Catalan numbers in the sense that $\Cat(W)=\sum_{k=0}^{r}\Nar(W,k)$.

The property of the Narayana numbers that will most concern us here is their \dfn{symmetry}: 
\begin{equation} \label{eq:nar_sym}
\Nar(W,k)=\Nar(W,r-k),
\end{equation}
for all $0\leq k \leq r$. This symmetry is easily seen using the noncrossing definition of the Narayana numbers: it follows from the fact that the lattice of noncrossing partitions is self-dual. (A particular duality of $\NC(W)$, the \dfn{Kreweras complement}, will feature prominently in what follows.) However, this symmetry is far from obvious using the nonnesting definition of the Narayana numbers. 

In this paper, we will take up the following problem:

\begin{problem} \label{prob:main}
Explain the symmetry~\eqref{eq:nar_sym} of the nonnesting $W$-Narayana numbers.
\end{problem}

\Cref{prob:main} was mentioned, for instance, in~\cite[Remark~5.10]{fomin2007rootsystems}.

\medskip

Our approach to~\cref{prob:main} will build upon a program of Panyushev~\cite{panyushev2004adnilpotent, panyushev2009orbits}. In particular, we will study the following conjecture of Panyushev:

\begin{conj}[{Panyushev~\cite[Conjecture~6.1]{panyushev2004adnilpotent}}] \label{conj:pan_intro}
There is a ``natural'' involution $\mathfrak{P}\colon \NN(W)\to\NN(W)$ for which $\#A + \#\mathfrak{P}(A)=r$ for all $A \in \NN(W)$.
\end{conj}

\begin{remark}
Let $\mathfrak{g}$ be the complex semisimple Lie algebra corresponding to the root system~$\Phi$, and let $\mathfrak{b}$ be the Borel subalgebra of $\mathfrak{g}$ corresponding to the choice of positive roots~$\Phi^+$. The nonnesting partitions $\NN(W)$ are in bijection with the \dfn{$\ad$-nilpotent ideals} in~$\mathfrak{b}$; under this bijection, the cardinality of the antichain becomes the minimal number of generators of the ideal. Hence, Panyushev described his conjectural~$\mathfrak{P}$ as a \dfn{duality} for $\ad$-nilpotent ideals of $\mathfrak{b}$, which sends an ideal with $k$ generators to one with $r-k$ generators.
\end{remark}

It is immediate from~\eqref{eq:nar_sym} that there is \emph{some} involution satisfying the condition in~\cref{conj:pan_intro}, so the word ``natural'' is doing all the work in this conjecture. Actually, Panyushev listed some specific desiderata for~$\mathfrak{P}$ that we will review later (see~\cref{conj:pan}). But, for instance, one thing we could want is that $\mathfrak{P}(A)$ is easily computable from $A$. Another thing we could want is that the definition of $\mathfrak{P}$ is purely ``poset-theoretic'' (only uses the poset structure of $\Phi^+$), since the definition of the Narayana numbers in terms of~$\NN(W)$ is poset-theoretic in this sense.

Panyushev was unable to define $\mathfrak{P}$ in general, but in~\cite{panyushev2004adnilpotent}, he was able to come up with a definition in Type~A. In fact, Panyushev's involution $\mathfrak{P}$ in Type~A is equivalent to the \dfn{Lalanne--Kreweras involution} on Dyck paths (see~\cite{hopkins2020birational}). A simple ``folding'' argument allows one to obtain the appropriate involution $\mathfrak{P}$ in Types~B/C from the one in Type~A, so Panyushev was also able to treat Types~B/C.

What's more, in a follow-up paper, Panyushev~\cite{panyushev2009orbits} conjectured a way to do something ``close'' to defining $\mathfrak{P}$ for all root systems. In that paper, he considered the \dfn{rowmotion} operator $\row \colon \NN(W)\to\NN(W)$ acting on nonnesting partitions, and he conjectured that it has very good behavior. Specifically, he conjectured that:
\begin{itemize}
\item $\row ^{h}$ is the involutive poset automorphism $-w_0\colon \Phi^+\to\Phi^+$, where $w_0\in W$ is the longest element (hence, $\row ^{2h}$ is the identity);
\item the average cardinality of the antichains in any $\row $-orbit is $\frac{r}{2}$.
\end{itemize}
This rowmotion conjecture is ``close'' to \cref{conj:pan_intro} because it says that $\NN(W)$ can be partitioned into blocks of sizes dividing $2h$ such that the average cardinality in each block is $\frac{r}{2}$, whereas \cref{conj:pan_intro} says that $\NN(W)$ can be partitioned into blocks of sizes dividing $2$ such that the average cardinality in each block is $\frac{r}{2}$. Panyushev's rowmotion conjecture was proved by Armstrong, Stump, and Thomas~\cite{armstrong2013uniform}.

Although Panyushev was not the first to consider rowmotion (which is defined more generally as acting on the antichains of any poset), his investigation of rowmotion on root posets rekindled interest in this operator. Indeed, the past 10 or so years has seen the emergence of the subfield of \dfn{dynamical algebraic combinatorics}~\cite{roby2016dynamical, striker2017dynamical}, in which rowmotion features prominently. Furthermore, Panyushev's ``constant average cardinality along orbits'' observation was one of the first instances of \dfn{homomesy}~\cite{propp2015homomesy}, a phenomenon that again is at the heart of dynamical algebraic combinatorics.

\medskip

In the present paper, we demonstrate how ideas from dynamical algebraic combinatorics can be used to address (more of)~\cref{conj:pan_intro}.

As shown by Cameron and Fon-der-Flaass~\cite{cameron1995orbits} (see also Joseph~\cite{joseph2019antichain}), rowmotion acting on the antichains of any poset can be written as a composition of \dfn{toggles}, which are certain simple involutions that add elements to or remove elements from subsets. For any ranked poset, there is another canonical composition of toggles that gives a related involutive operator called \dfn{rowvacuation}. The names ``rowmotion'' and ``rowvacuation'' come from Sch\"{u}tzenberger's promotion and evacuation operators on linear extensions. Rowmotion and rowvacuation share many properties with promotion and evacuation: for instance, they generate a dihedral group action.

Rowvacuation is easy to compute (because it is a composition of toggles), and its definition is inherently poset-theoretic. Our main result is the following:

\begin{thm} \label{thm:main_intro}
If $\Phi$ is one of the classical types~A, B, C, or~D, then the rowvacuation operator $\Rvac\colon \NN(W)\to\NN(W)$ is Panyushev's conjectured involution $\mathfrak{P}$ from~\cref{conj:pan_intro}.
\end{thm}

For $\Phi=G_2$, it is easy to see there is a unique choice of $\mathfrak{P}$, which in fact agrees with rowvacuation. Unfortunately, for the other exceptional types $E_6$, $E_7$, $E_8$, and~$F_4$, there are antichains $A \in \NN(W)$ with $\#A+\#\Rvac(A)\neq r$ (see~\cref{rem:exceptional}), so we are unable to resolve~\cref{conj:pan_intro} in these exceptional types. 

It was recently shown by the second author and Joseph~\cite{hopkins2020birational} that the Lalanne--Kreweras involution is rowvacuation for the root posets of Type A. Together with Panyushev's prior work from~\cite{panyushev2004adnilpotent}, this proves~\cref{thm:main_intro} for Types~A, B, and~C. So the only case we have to address here is Type D.\footnote{But note that Type D is usually the hardest: for instance, the original definition of noncrossing partitions in Type D was ``incorrect''~\cite{athanasiadis2004noncrossing, bessis2003dual, reiner1997noncrossing}.} However, along the way, we prove results concerning rowvacuation of $\NN(W)$ for arbitrary $\Phi$. 

More precisely, we extend some results of Armstrong--Stump--Thomas~\cite{armstrong2013uniform}. In order to resolve Panyushev's rowmotion conjecture, as well as a related conjecture of Bessis and Reiner~\cite{bessis2011cyclic}, they constructed a bijection $\Theta_W\colon \NN(W)\to \NC(W)$ between the nonnesting and noncrossing partitions that is uniquely specified by a handful of properties (see~\cref{thm:ast}). The most important of these properties is that 
\[\Theta_W \cdot \row= \Krew\cdot \Theta_W,\] where $\Krew\colon \NC(W)\to\NC(W)$ denotes the \dfn{Kreweras complement} defined by $\Krew(w) \coloneqq cw^{-1}$. We show:

\begin{thm} \label{thm:ast_intro}
For the bijection~$\Theta_W\colon \NN(W)\to \NC(W)$ from~\cite{armstrong2013uniform}, we have 
\[\Theta_W \cdot (\row^{-1} \cdot \Rvac) = \Flip\cdot \Theta_W,\] 
where $\Flip\colon \NC(W)\to\NC(W)$ is the involutive poset automorphism defined by $\Flip(w) \coloneqq gw^{-1}g^{-1}$ (for the appropriate involution $g\in W$ depending on $c$).
\end{thm}

Note that we prove \cref{thm:ast_intro} just by appealing to the general properties satisfied by~$\Theta_W$, and not via any case-by-case reasoning. Nevertheless, a lot of intricate combinatorial reasoning particular to the classical types does go into our proof of~\cref{thm:main_intro}. For example, in these classical types, there are models of noncrossing partitions where $\Krew$ corresponds to rotation; we show that $\Flip$ corresponds to reflection across a diameter in these models (hence the name).

\medskip

\begin{remark} 
Before we end this introduction, let us briefly discuss an alternate approach to~\cref{prob:main} that we will not pursue. This alternate approach concerns an even further refinement of the Catalan numbers. Let $\Pi(\Phi)$ denote the lattice of flats of the Coxeter arrangement of $\Phi$. There are natural embeddings
\begin{align*}
\iota_{\NN}\colon \NN(W) &\hookrightarrow \Pi(\Phi);\\
\iota_{\NC}\colon \NC(W) &\hookrightarrow \Pi(\Phi).
\end{align*} 
Define the \dfn{type}\footnote{This notion of type is unrelated to the Cartan--Killing type of the root system~$\Phi$. This clash in terminology is unfortunate but standard.} of $A\in \NN(W)$ to be the $W$-orbit of $\iota_{\NN}(A)$, and define the type of $w \in \NC(W)$ similarly. It is known~\cite[Theorem~6.3]{athanasiadis2004noncrossing} that for any $X \in \Pi(\Phi)/W$, the number of nonnesting partitions of type~$X$ is the same as the number of noncrossing partitions of type~$X$. These numbers are called the \dfn{$W$-Kreweras numbers}.\footnote{We apologize for the number of different (and often unrelated) mathematical concepts in our paper named after Kreweras. Although, to be fair, it is mostly he who is to blame.} Moreover, the codimension of $\iota_{\NN}(A)$ is the same as the cardinality of $A \in \NN(W)$, and the codimension of $\iota_{\NC}(w)$ is the same as the rank of $w \in \NC(W)$. Since, as mentioned above, the symmetry of the Narayana numbers is easy to see in the noncrossing world, a type-preserving bijection between nonnesting and noncrossing partitions would, therefore, constitute a solution to~\cref{prob:main}. There are a handful of type-preserving bijections between the nonnesting and noncrossing partitions in the literature~\cite{kim2011new, fink2009bijections}; but these constructions are all ad-hoc and, thus, unsatisfactory in some sense. No uniform type-preserving bijection is known.
\end{remark}

\medskip

The rest of the paper is structured as follows. In \cref{sec:posets}, we review necessary background concerning root posets, rowmotion and rowvacuation, and Panyushev's~\cref{conj:pan_intro}. In \cref{sec:Flip}, we review noncrossing partitions, the Kreweras complement, and the Armstrong--Stump--Thomas bijection. We also introduce the $\Flip$ operator and prove~\cref{thm:ast_intro} there. Armstrong, Stump, and Thomas gave explicit combinatorial descriptions of their bijections in classical types. We review their description in Type~A in \cref{Sec:ASTBijectionA} and review parts of their description in Type~D in \cref{Sec:ASTBijectionD}. These sections also establish some lemmas, concerning how rowvacuation interacts with these combinatorial constructions, that will be needed in the proof of \cref{thm:main_intro}. In \cref{Sec:MainProof}, we tie up the remaining loose ends and prove that rowvacuation serves as Panyushev's $\mathfrak{P}$ in Type~D, thus completing the proof of~\cref{thm:main_intro}. 

\medskip

\noindent {\bf Acknowledgments}: We thank the organizers of the 2020 BIRS Online Workshop on dynamical algebraic combinatorics for giving us the opportunity to collaborate on this project. We thank Theo Douvropoulos, Michael Joseph, Hugh Thomas, and Nathan Williams for useful discussion during that workshop. S.H.~gives a special thanks to Nathan Williams for suggesting the correct involution $g$ to use in the definition of $\Flip$ for root systems of arbitrary type. 

\section{Posets}\label{sec:posets}

\subsection{Root posets} \label{subsec:RootPosets}

We assume the reader is familiar with the basics concerning posets as laid out, for instance, in~\cite[Chapter~3]{stanley2011ec1}. All posets we consider in this paper are finite, and we drop this adjective from now on. We say a poset $P$ is \dfn{ranked} if there exists a \dfn{rank function} $\rk\colon P \to \mathbb{N}$ for which
\begin{itemize}
\item $\rk(p)=0$ for all minimal elements of $P$;
\item $\rk(y)=\rk(x)+1$ if $x,y\in P$ are such that $x \lessdot y$.
\end{itemize}
A rank function is unique if it exists. The \dfn{rank} $\rk(P)$ of a ranked poset $P$ is the maximum value of its rank function. From now on, all posets we consider will be ranked. For $i\in \mathbb{N}$, we use the notation
\[P_i\coloneqq \{p\in P\colon \rk(p)=i\}.\]
Note that $P_i$ is empty unless $0\leq i \leq \rk(P)$. The nonempty $P_i$ are called the \dfn{ranks} of $P$. We use $\mathcal{A}(P)$ to denote the set of \dfn{antichains} of $P$. Evidently, $P_i\in \mathcal{A}(P)$ for all $i \in \mathbb{N}$.

As in \cref{sec:intro}, let $\Phi$ be a crystallographic root system in an $r$-dimensional vector space $V$ with Weyl group $W\subseteq\GL(V)$. For basic background on root systems, see, e.g.,~\cite{bourbaki2002lie}. We assume we have chosen a system of \dfn{simple roots} $\{\alpha_1,\ldots,\alpha_r\}$ and, hence, a corresponding set~$\Phi^+$ of \dfn{positive roots} (those that are nonnegative linear combinations of the simple roots). The \dfn{root poset} of $\Phi$ is the partial order on $\Phi^+$ whereby $\alpha\leq \beta$ if $\beta-\alpha$ is a nonnegative sum of simple roots. By abuse of notation, we use $\Phi^+$ to denote the root poset; if $\Phi$ has type $X$, we also use $\Phi^+(X)$ to denote this poset. The root poset is ranked with the rank function being $\rk(\alpha)=\height(\alpha)-1$, where the \dfn{height} of a positive root $\alpha = \sum_{i}a_i\alpha_i \in \Phi^+$ is $\height(\alpha)\coloneqq \sum_{i}a_i$. The minimal elements of $\Phi^+$ are the simple roots. If $\Phi$ is irreducible, then it has a unique maximal element called the \dfn{highest root}.  Again, if $\Phi$ is irreducible, then the rank of $\Phi^+$ is $\rk(\Phi^+)=h-2$, where we recall that $h$ denotes the Coxeter number of~$W$. (In fact, \emph{all} the degrees of $W$ can similarly be read off from the sizes of ranks of $\Phi^+$.) The poset~$\Phi^+$ has a canonical involutive automorphism $\alpha\mapsto -w_0(\alpha)$, where $w_0\in W$ is the \dfn{longest element}.

As discussed in~\cref{sec:intro}, we call the antichains of $\Phi^+$ the \dfn{$W$-nonnesting partitions}. This name is due to Postnikov, who first suggested studying them as Catalan objects (see~\cite[Remark~2]{reiner1997noncrossing}).  In \cref{sec:intro}, we denoted this set of antichains by~$\NN(W)$, but from now on, we will use the notation $\mathcal{A}(\Phi^+)$.

We now describe specific realizations of the root posets of classical-type root systems. In anticipation of our detailed analysis of these posets, we will also highlight connections between them and define some subsets $\mathcal{L},\mathcal{S} \subseteq \Phi^+$ of particular interest. We use the standard notation $[i,j] \coloneqq \{i,i+1,\ldots,j\}$ for intervals, and we also use~$[n]\coloneqq [1,n]$ and $-[n]\coloneqq [-n,-1]$. For $1\leq i \leq n$, let $e_i$ denote the $i^\text{th}$ standard basis vector in~$\mathbb R^n$. 

The elements of the root poset $\Phi^+(A_{n-1})$ are $e_i-e_j$ for $1\leq i < j \leq n$. The simple roots $\alpha_1,\ldots,\alpha_{n-1}$ are $\alpha_i=e_i-e_{i+1}$. We identify $\Phi^+(A_{n-1})$ with the set of intervals $\{[i,j]\subseteq [n]\colon 1 \leq i < j \leq n\}$ ordered by containment, where the root $e_i-e_j$ corresponds to the interval $[i,j]$. \Cref{fig:typeAposet} presents the Hasse diagram of $\Phi^+(A_9)$. The poset~$\Phi^+(A_{n-1})$ has an involutive automorphism $\eta:\Phi^+(A_{n-1})\to\Phi^+(A_{n-1})$ given by $\eta([i,j])=[n+1-j,n+1-i]$, which, simply put, is reflection of the Hasse diagram across the central vertical axis. In fact, $\eta=-w_0$. We denote the set of elements of $\Phi^+(A_{n-1})$ that are fixed by $\eta$ by $\mathcal L_{A_{n-1}}\coloneqq \{[i,n+1-i]\colon1\leq i\leq \left\lfloor n/2\right\rfloor\}$.

\begin{figure}
\begin{center}
{\includegraphics[height=3.724cm]{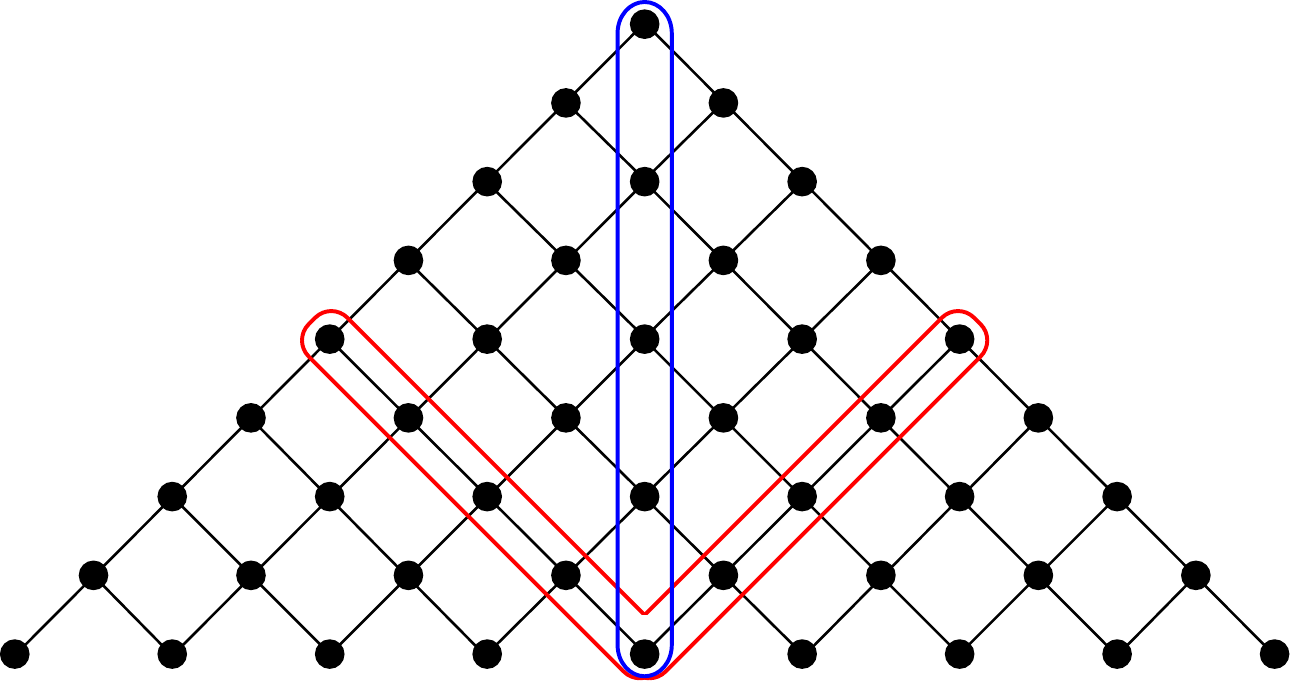}}
\end{center}
\caption{The root poset $\Phi^+(A_9)$. The minimal elements are the simple roots, which can be identified with the intervals $[1,2],[2,3],\ldots,[9,10]$ (ordered as they appear from left to right in the Hasse diagram). The set $\mathcal L_{A_{9}}$ consists of the $5$ elements circled in blue. The set $\mathcal{S}_{A_{9}}$ consists of the $9$ elements circled in red.}\label{fig:typeAposet}
\end{figure}

The root posets of Types~B and~C are isomorphic, so we will focus our attention on Type~C. The elements of $\Phi^+(C_n)$ are $e_i\pm e_j$ for $1\leq i<j\leq n$, together with~$2e_i$ for $1\leq i \leq n$. The simple roots $\alpha_1,\ldots,\alpha_n$ are \[ \alpha_i = \begin{cases} e_i - e_{i+1} &\textrm{if $1\leq i \leq n-1$}; \\
2e_n &\textrm{if $i=n$}.\end{cases}\]  The poset~$\Phi^+(C_n)$ can be realized as the quotient of $\Phi^+(A_{2n-1})$ by the action of $\eta$. In other words, the Hasse diagram of $\Phi^+(C_n)$ is obtained by ``folding''~$\Phi^+(A_{2n-1})$ along its central vertical axis. \Cref{fig:typeCposet} presents the Hasse diagram of $\Phi^+(C_5)$. There is a natural injection $\iota\colon \mathcal P(\Phi^+(C_n))\to\mathcal P(\Phi^+(A_{2n-1}))$ obtained by ``unfolding,'' where~$\mathcal{P}(X)$ denotes the power set of $X$. Thus, the image of $\iota$ is the collection of subsets of $\Phi^+(A_{2n-1})$ that are symmetric about the central vertical axis. We denote by $\mathcal L_{C_n}$ the set of long roots in $\Phi^+(C_n)$, which correspond to the singleton $\eta$-orbits in $\Phi^+(A_{2n-1})$ (i.e., the elements lying on the ``crease of the fold''). Equivalently, we have~$\mathcal L_{C_n}=\iota^{-1}(\mathcal L_{A_{2n-1}})$. 

\begin{figure}
\begin{center}
{\includegraphics[height=3.724cm]{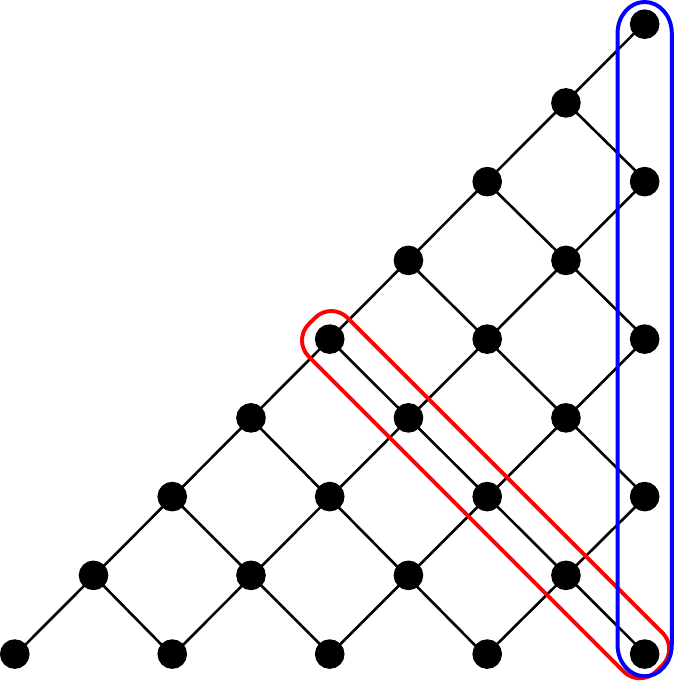}}
\end{center}
\caption{The root poset $\Phi^+(C_5)$. The set $\mathcal L_{C_5}$ consists of the $5$ elements circled in blue. The set $\mathcal{S}_{C_5}$ consists of the $5$ elements circled in red.}\label{fig:typeCposet}
\end{figure}

The elements of the root poset $\Phi^+(D_n)$ are $e_i\pm e_j$ for $1\leq i<j\leq n$. The simple roots $\alpha_1,\ldots,\alpha_n$ are
\[ \alpha_i = \begin{cases} e_i - e_{i+1} &\textrm{if $1\leq i \leq n-1$}; \\
e_{n-1}+e_n &\textrm{if $i=n$}.\end{cases}\]
See~\cref{fig:typeDposet} for a depiction of $\Phi^+(D_6)$. There is an involutive automorphism $\delta$ of the Dynkin diagram of $\Phi(D_n)$ that swaps the nodes $n-1$ and $n$. This induces an automorphism of the poset $\Phi^+(D_n)$ that swaps every occurrence of $\alpha_{n-1}$ appearing in an expansion of a root with $\alpha_n$ and vice versa; we also denote this automorphism by $\delta\colon \Phi^+(D_n) \to \Phi^+(D_n)$. We have $\delta(e_i-e_n)=e_i+e_n$ and $\delta(e_i+e_n)=e_i-e_n$ for each $1\leq i\leq n-1$; all of the other positive roots are fixed by~$\delta$. (Note that~$\delta=-w_0$ when $n$ is odd; but~$-w_0$, unlike $\delta$, is trivial when $n$ is even.)  In \cref{fig:typeDposet}, the roots of the form $e_i+e_6$ are colored white. Furthermore, for each~$1\leq i\leq 5$, we draw the root $e_i-e_6$ immediately to the right of $e_i+e_6$ in the figure. There is a natural bijection between the $\delta$-orbits in $\Phi^+(D_n)$ and the elements of $\Phi^+(C_{n-1})$, yielding the quotient map~$\gamma:\Phi^+(D_n)\to\Phi^+(C_{n-1})$. Referring to \cref{fig:typeDposet} again, we see that~$\gamma$ essentially ``glues'' each white element to the black element drawn immediately to the right of it. Let us define $\mathcal L_{D_n}\coloneqq\gamma^{-1}(\mathcal L_{C_{n-1}})$. 

\begin{figure}
\begin{center}
{\includegraphics[height=3.8cm]{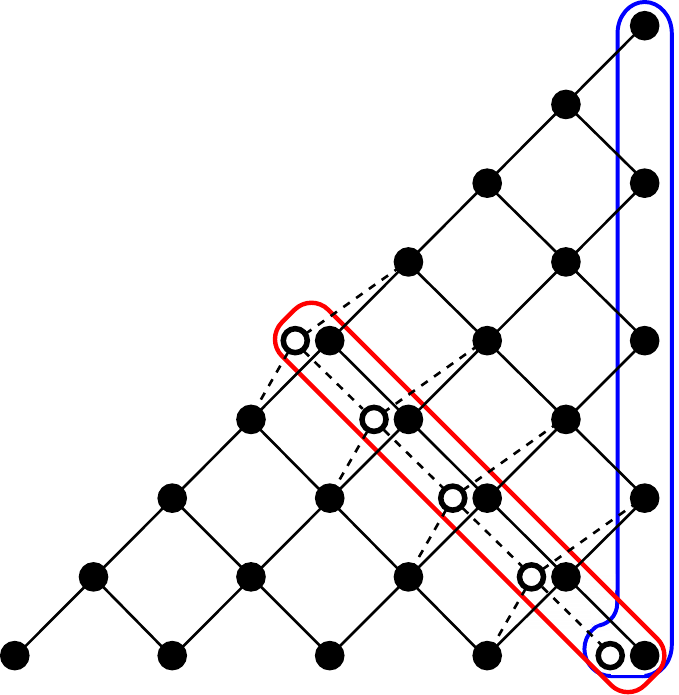}}
\end{center}
\caption{The root poset $\Phi^+(D_6)$. The set $\mathcal L_{D_6}$ consists of the $6$ elements circled in blue. The set $\mathcal{S}_{D_6}$ consists of the $10$ elements circled in red.}\label{fig:typeDposet}
\end{figure}

We define $\mathcal{S}_{C_n}\subseteq \Phi^+(C_n)$ to be the set of elements $\alpha\in\Phi^+(C_n)$ for which~$\gamma^{-1}(\alpha)$ consists of two elements; all other $\alpha\in\Phi^+(C_n)$ have $\#\gamma^{-1}(\alpha)=1$. (Under the isomorphism $\Phi^+(C_n)\simeq\Phi^+(B_n)$, $\mathcal{S}_{C_n}$ consists of the short roots in $\Phi^+(B_n)$.) We also define $\mathcal{S}_{A_{2n-1}}\subseteq \Phi^+(A_{2n-1})$ to be $\iota(\mathcal{S}_{C_n})$, and $\mathcal{S}_{D_n}\subseteq \Phi^+(D_n)$ to be $\gamma^{-1}(\mathcal{S}_{C_{n-1}})$. Equivalently, $\mathcal{S}_{A_{2n-1}}$ consists of those $[i,j] \in \Phi^+(A_{2n-1})$ for which either $i=n$ or $j=n+1$. These subsets $\mathcal{S}$ are also depicted in~\cref{fig:typeAposet,fig:typeCposet,fig:typeDposet}.

\begin{remark}
The processes of obtaining $\Phi^+(B_n)$ ($\simeq \Phi^+(C_n)$) as a quotient of $\Phi^+(A_n)$ and of obtaining $\Phi^+(C_n)$ as a quotient of~$\Phi^+(D_n)$ are special cases of a more general procedure from the theory of root systems referred to as \dfn{folding}. Given a simply laced root system~$\Phi$ and an automorphism $\sigma$ of the Dynkin diagram of $\Phi$ (subject to a certain technical condition), folding produces a new root system $\Phi^{\sigma}$ whose roots correspond to $\sigma$-orbits. Since we only need the cases of folding we have described explicitly above, we will not give a precise account of folding here; see~\cite{stembridge2008folding} for such an account. We also warn that while $\mathcal{L}$ and $\mathcal{S}$ were chosen to evoke the words ``long'' and ``short,'' the simply laced types A and D of course do not have roots of different lengths, so this notation is just meant to be suggestive.
\end{remark}

\subsection{Rowmotion and rowvacuation}

Now we review the rowmotion and rowvacuation operators acting on the antichains of a ranked poset $P$.\footnote{In~\cite{hopkins2020order} and~\cite{hopkins2020birational}, rowvacuation was defined for \dfn{graded} posets. The hypothesis that $P$ is graded is slightly stronger than the hypothesis that it is ranked: it requires additionally that $\rk(p)=\rk(P)$ for all maximal elements $p$ of $P$. However, none of the proofs of the basic properties of rowvacuation require this stronger assumption.}

\dfn{Rowmotion}, $\row \colon\mathcal{A}(P)\to\mathcal{A}(P)$, is given by
\[\row (A) \coloneqq \min(\{x\in P\colon \textrm{$x \not\leq y$ for all $y\in A$}\}),\] 
for all $A\in\mathcal{A}(P)$, where $\min(X)$ means the set of minimal elements of a subset $X\subseteq P$. Rowmotion is a bijection. Before Panyushev~\cite{panyushev2004adnilpotent}, rowmotion was studied by Brouwer and Schrijver~\cite{brouwer1974period} and Cameron and Fon-der-Flaass~\cite{cameron1995orbits}, among others.

\begin{example}
The two orbits of rowmotion on $\mathcal{A}(\Phi^+(A_2))$ are:
\begin{center}
\includegraphics[height=1.077cm]{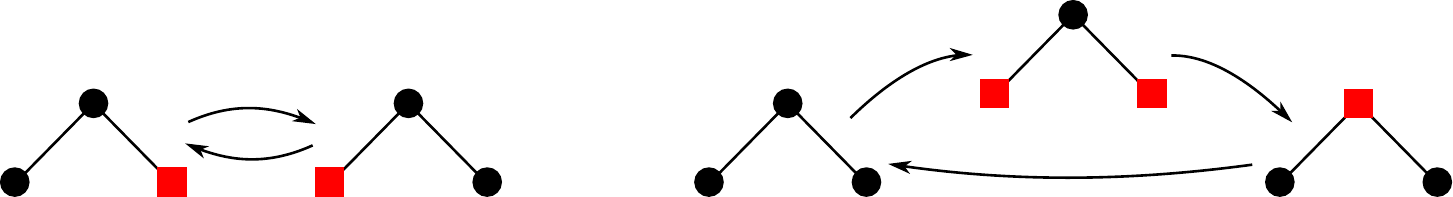}
\end{center}
The elements of each antichain are depicted with red squares. Observe $\row ^3=\eta$ and the average of $\#A$ along any $\row $-orbit is $1$.
\end{example}

As discussed in~\cref{sec:intro}, Armstrong, Stump, and Thomas~\cite{armstrong2013uniform} proved the following conjecture of Panyushev concerning rowmotion of nonnesting partitions.

\begin{thm}[{Armstrong--Stump--Thomas~\cite[Theorem~1.2]{armstrong2013uniform}}] \label{thm:ast_row}
The rowmotion operator $\row \colon \mathcal{A}(\Phi^+)\to \mathcal{A}(\Phi^+)$ satisfies the following:
\begin{itemize}
\item $\row ^{2h}$ is the identity (and $\row ^h=-w_0$ if $\Phi$ is irreducible);
\item $\frac{1}{\#O} \sum_{A \in O} \#A = \frac{r}{2}$ for any $\row $-orbit $O\subseteq \mathcal{A}(\Phi^+)$.
\end{itemize}
\end{thm}

As shown by Cameron and Fon-der-Flaass~\cite{cameron1995orbits}, and further emphasized by Striker and Williams~\cite{striker2012promotion}, there is an alternate way to describe rowmotion as a composition of certain local involutions called \dfn{toggles}. However, in~\cite{cameron1995orbits} and~\cite{striker2012promotion}, the order ideal variant of rowmotion and toggles are considered. We prefer to stick to antichains, and hence will focus on a description due to Joseph~\cite{joseph2019antichain} of rowmotion as a composition of antichain toggles. For $p\in P$, define the \dfn{toggle at $p$}, $\tau_p\colon \mathcal{A}(P)\to \mathcal{A}(P)$, to be the involution
\[\tau_p(A)\coloneqq \begin{cases} A\setminus \{p\} &\textrm{if $p\in A$}; \\
A\cup \{p\} &\textrm{if $p \notin A$ and $A\cup \{p\} \in \mathcal{A}(P)$};\\
A &\textrm{otherwise}.\end{cases}\]
Let us emphasize that these antichain toggles are not the same as order ideal toggles. Also, if $p,p'\in P$ are incomparable, then the toggles $\tau_p$ and $\tau_{p'}$ commute. 

\begin{prop}[{Joseph~\cite[Proposition~ 2.24]{joseph2019antichain}}]
We have $\row =\tau_{p_n} \tau_{p_{n-1}}\cdots \tau_{p_1}$ for any linear extension $p_1,p_2,\ldots,p_n$ of $P$.
\end{prop}

Next we define rowvacuation, which, unlike rowmotion, requires our poset to be ranked. For $i=0,1,\ldots,\rk(P)$, we define the \dfn{rank toggle} $\boldsymbol{\tau}_i\colon\mathcal{A}(P)\to \mathcal{A}(P)$ to be~$\boldsymbol{\tau}_i \coloneqq \prod_{p \in P_i} \tau_p$. All the toggles $\tau_p$ for $p\in P_i$ commute, so this product makes sense. Evidently, 
\[\row= \boldsymbol{\tau}_{\rk(P)} \boldsymbol{\tau}_{\rk(P)-1} \cdots \boldsymbol{\tau}_0.\] 
\dfn{Rowvacuation}, $\Rvac\colon \mathcal{A}(P) \to\mathcal{A}(P)$, is a different product of these rank toggles:
\[ \Rvac\coloneqq (\boldsymbol{\tau}_{\rk(P)}) (\boldsymbol{\tau}_{\rk(P)} \boldsymbol{\tau}_{\rk(P)-1}) \cdots (\boldsymbol{\tau}_{\rk(P)} \boldsymbol{\tau}_{\rk(P)-1} \cdots \boldsymbol{\tau}_1) (\boldsymbol{\tau}_{\rk(P)} \boldsymbol{\tau}_{\rk(P)-1} \cdots \boldsymbol{\tau}_0). \] 
We write $\row_P$ and $\Rvac_P$ when we wish to emphasize the underlying poset $P$; this will be useful when we consider multiple posets at once. 

\begin{example}
We show a computation of $\Rvac(A)$ for an antichain $A\in\mathcal{A}(\Phi^+(D_4))$ below:
\[\hphantom{\xrightarrow{\boldsymbol{\tau}_4\boldsymbol{\tau}_3\boldsymbol{\tau}_2\boldsymbol{\tau}_1\boldsymbol{\tau}_0}}\begin{array}{c}
\includegraphics[height=1.887cm]{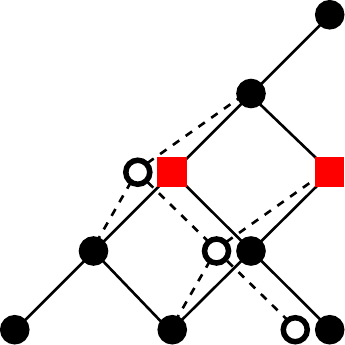}
\end{array}\xrightarrow{\boldsymbol{\tau}_4\boldsymbol{\tau}_3\boldsymbol{\tau}_2\boldsymbol{\tau}_1\boldsymbol{\tau}_0}\begin{array}{c}
\includegraphics[height=1.887cm]{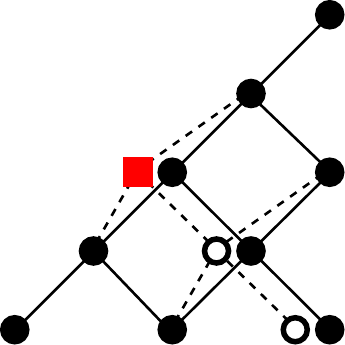}
\end{array}\xrightarrow{\,\,\,\boldsymbol{\tau}_4\boldsymbol{\tau}_3\boldsymbol{\tau}_2\boldsymbol{\tau}_1\,\,\,}\begin{array}{c}
\includegraphics[height=1.887cm]{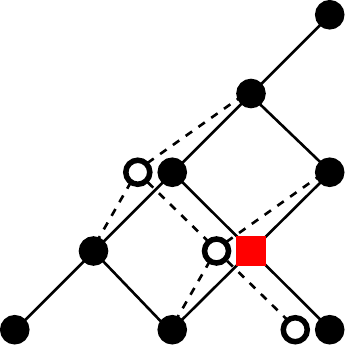}
\end{array}\]
\[\xrightarrow{\hphantom{\boldsymbol{\tau}_1}\boldsymbol{\tau}_4\boldsymbol{\tau}_3\boldsymbol{\tau}_2\hphantom{\boldsymbol{\tau}_0}}\begin{array}{c}
\includegraphics[height=1.887cm]{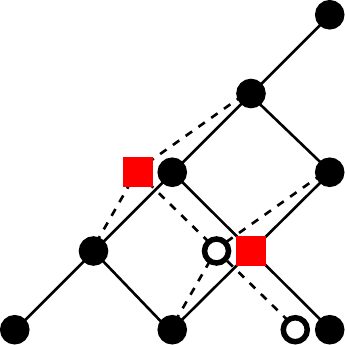}
\end{array}\xrightarrow{\,\,\,\hphantom{\boldsymbol{\tau}_1}\boldsymbol{\tau}_4\boldsymbol{\tau}_3\hphantom{\boldsymbol{\tau}_0}\,\,\,}\begin{array}{c}
\includegraphics[height=1.887cm]{RowvacPIC24New}
\end{array}\xrightarrow{\hphantom{\boldsymbol{\tau}_3\boldsymbol{\tau}_2}\boldsymbol{\tau}_4\hphantom{\boldsymbol{\tau}_1\boldsymbol{\tau}_0}}\begin{array}{c}
\includegraphics[height=1.887cm]{RowvacPIC24New}
\end{array}.\]
Observe that $\#A+\#\Rvac(A)=2+2=4$.
\end{example}

Rowvacuation was first formally defined, briefly, in~\cite[\S5.1]{hopkins2020order} and was further studied in~\cite{hopkins2020birational}. As mentioned in~\cref{sec:intro}, rowmotion and rowvacuation are ``partner'' operators in exactly the same way that  Sch\"{u}tzenberger's promotion and evacuation operators are. For instance, together, they always generate a dihedral group action:

\begin{prop}[{\cite[Proposition~5.1]{hopkins2020order} \cite[Proposition~2.18]{hopkins2020birational}}] \label{prop:row_basics}
For any ranked~$P$:
\begin{itemize}
\item $\Rvac$ is an involution;
\item $\Rvac\cdot \row= \row^{-1} \cdot \Rvac$.
\end{itemize}
\end{prop}

In general, it seems hard to ``describe'' the rowvacuation of an antichain. But in~\cite{hopkins2020birational}, the second author and Joseph showed that rowvacuation acting on a root poset of Type~A can be computed as follows. 

\begin{thm}[{H.--Joseph~\cite[Theorem~3.5]{hopkins2020birational}}] \label{thm:lk}
Let $A\in\mathcal{A}(\Phi^+(A_{n-1}))$. Note $A$ is of the form $A=\{[i_1,j_1],\ldots,[i_k,j_k]\}$ with $i_1 < i_2 < \cdots < i_k$ and $j_1 < j_2 < \cdots < j_k$. Then
\[\Rvac(A) = \{ [i'_1,j'_1],\ldots, [i'_{n-1-k},j'_{n-1-k}]\},\]
where
\begin{align}  
\notag \{i'_1 < i'_2 < \cdots < i'_{n-1-k}\} &\coloneqq \{1,2,\ldots,n-1\}\setminus \{j_1-1,j_2-1,\ldots,j_k-1\}, \\
\notag \{j'_1 < j'_2 < \cdots < j'_{n-1-k}\} &\coloneqq \{2,3,\ldots,n\}\setminus \{i_1+1,i_2+1,\ldots,i_k+1\}.
\end{align}
\end{thm}

It is immediate from the formula in~\cref{thm:lk} that $\#A+\#\Rvac(A)=n-1$. As explained in~\cite[\S1]{hopkins2020birational}, this formula is equivalent to the \dfn{Lalanne--Kreweras involution} on Dyck paths via a simple bijection between the antichains in $\mathcal{A}(\Phi^+(A_{n-1}))$ and the Dyck paths of length $2n$. Rowvacuation clearly commutes with any poset automorphism. Therefore, by embedding $\mathcal{A}(\Phi^+(C_n))$ into $\mathcal{A}(\Phi^+(A_{2n-1}))$ via $\iota$, \cref{thm:lk} also yields a simple description of rowvacuation acting on $\mathcal{A}(\Phi^+(C_n))$.

\begin{example}
If $A=\{[1,3]\}\in\mathcal{A}(\Phi^+(A_3))$, then $\Rvac(A)=\{[1,3],[3,4]\}$:
\begin{center}
\includegraphics[height=1.023cm]{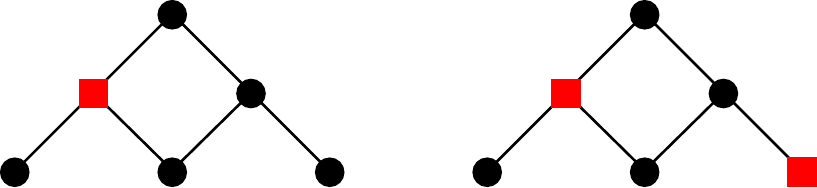}
\end{center}
The reader may check that this agrees with~\cref{thm:lk}
\end{example}

Before we move on to discuss Panyushev's conjectured duality for root poset antichains, let us prove a few more general properties of rowvacuation that will show that it is a good candidate for this duality.

\begin{prop} \label{prop:rvac_ranks}
We have $\Rvac(P_i)=P_{\rk(P)+1-i}$ for any $i=0,1,\ldots,\rk(P)+1$.
\end{prop}
\begin{proof}
Define $d_k \coloneqq (\boldsymbol{\tau}_{\rk(P)} \boldsymbol{\tau}_{\rk(P)-1} \cdots \boldsymbol{\tau}_k)$ for $k=0,\ldots,\rk(P)$ so that
\[ \Rvac= d_{\rk(P)} d_{\rk(P)-1} \cdots d_1 d_0.\]
It is easy to see for $i=0,\ldots,\rk(P)+1$ that
\[ d_k (P_i) = \begin{cases} P_{k} &\textrm{if $i=\rk(P)+1$}; \\
P_{i+1} &\textrm{if $k\leq i\leq \rk(P)$}; \\
P_i &\textrm{otherwise}. \end{cases} \]
Therefore,
\begin{align*}
\Rvac(P_i) &= d_{\rk(P)} d_{\rk(P)-1} \cdots d_1 d_0(P_i) \\
&= d_{\rk(P)} d_{\rk(P)-1} \cdots d_{\rk(P)+1-i} (P_{\rk(P)+1}) \\
&= d_{\rk(P)} d_{\rk(P)-1} \cdots d_{\rk(P)+2-i} (P_{\rk(P)+1-i}) \\
&= P_{\rk(P)+1-i},
\end{align*}
as claimed.
\end{proof}

\begin{prop} \label{prop:rvac_parabolic}
Let $p \in P_0$ be a minimal element, and let $P' \coloneqq \{q\in P\colon q \not\geq p\}$. Then for any $A \in \mathcal{A}(P)$:
\begin{itemize}
\item if $A\subseteq P'$, then $\Rvac(A) = \{p\} \cup \Rvac_{P'}(A)$;
\item if $p \in A$, then $\Rvac(A) = \Rvac_{P'}(A\setminus\{p\})$.
\end{itemize}
\end{prop}
\begin{proof}
Let us prove the first bulleted item. If $A\subseteq P'$, then $p \not \in A$. Hence, when we carry out the rank toggles
\[(\boldsymbol{\tau}_{\rk(P)}) (\boldsymbol{\tau}_{\rk(P)} \boldsymbol{\tau}_{\rk(P)-1}) \cdots (\boldsymbol{\tau}_{\rk(P)} \boldsymbol{\tau}_{\rk(P)-1} \cdots \boldsymbol{\tau}_1) (\boldsymbol{\tau}_{\rk(P)} \boldsymbol{\tau}_{\rk(P)-1} \cdots \boldsymbol{\tau}_0),\]
the application of $\boldsymbol{\tau}_0$ will add $p$ to $A$. From then on, no $q \geq p$ can be added to $A$. Therefore, the effect of carrying out this sequence of rank toggles will be the same as if we carried them out just on $P'$, except that we also have to add $p$. This is exactly what the equality $\Rvac(A) = \{p\} \cup \Rvac_{P'}(A)$ asserts.

For the second item: if $p\in A$, then $A$ is of the form $A=\{p\}\cup A'$ for some antichain $A' \subseteq P'$. So this item actually follows from the first item and the fact that $\Rvac$ is an involution (see Proposition~\ref{prop:row_basics}).
\end{proof}

\subsection{Panyushev's nonnesting partition duality conjecture} \label{subsec:pan_conj}

We now give a more precise version of~\cref{conj:pan_intro}.

\begin{conj}[{Panyushev~\cite[Conjecture~6.1]{panyushev2004adnilpotent}}] \label{conj:pan}
There is an involution on the set of nonnesting partitions $\mathfrak{P}\colon\mathcal{A}(\Phi^+)\to\mathcal{A}(\Phi^+)$ satisfying the following properties. First of all, if $\Phi=\Phi'\sqcup\Phi''$ is reducible, then $\mathfrak{P}(A)=\mathfrak{P}_{\Phi'}(A\cap(\Phi')^+)\cup\mathfrak{P}_{\Phi''}(A\cap(\Phi'')^+)$ for all~$A \in \mathcal{A}(\Phi^+)$. If $\Phi$ is irreducible, then for all~$A\in \mathcal{A}(\Phi^+)$, we have:
\begin{enumerate}[(i)]
\item $\#A+\#\mathfrak{P}(A)=r$;
\item the distribution of long and short roots in the multiset union $A\cup \mathfrak{P}(A)$ is the same as in the set of simple roots~$\{\alpha_1,\ldots,\alpha_r\}$;
\item if $A = \Phi^+_i$, then $\mathfrak{P}(A)=\Phi^+_{h-1-i}$ for all $i=0,\ldots,h-1$;
\item if $\alpha \in A$ for a simple root $\alpha$, then $\mathfrak{P}(A)=\mathfrak{P}_{\Phi'}(A\setminus \{\alpha\})$, where $\Phi'\subseteq \Phi$ is the maximal parabolic sub-root system of $\Phi$ obtained by removing $\alpha$ from the system of simple roots;
\item if $A \subseteq (\Phi')^+$, where $\Phi'\subseteq \Phi$ is the maximal parabolic sub-root system of $\Phi$ obtained by removing some simple root $\alpha$, then $\mathfrak{P}(A)=\{\alpha\}\cup\mathfrak{P}_{\Phi'}(A)$.
\end{enumerate}
\end{conj}

Remember that we are trying to show rowvacuation is Panyushev's~$\mathfrak{P}$: \cref{thm:main_intro} asserts that $\mathfrak{P}=\Rvac$ for the classical types. So let us check which of the properties in~\cref{conj:pan} rowvacuation satisfies. $\Rvac$ evidently respects the decomposition of a root system into its irreducible components, which at the level of root posets corresponds to the decomposition into connected posets. Furthermore, property~(iii) for $\Rvac$ is~\cref{prop:rvac_ranks}, and properties~(iv) and~(v) are~\cref{prop:rvac_parabolic}. So the only properties that we do not yet know $\Rvac$ satisfies are~(i) and~(ii). Of course, (i) is the property we really care about: it says that $\mathfrak{P}$ combinatorially exhibits the symmetry of the $W$-Narayana numbers.

In~\cite[\S4]{panyushev2004adnilpotent}, Panyushev showed that defining $\mathfrak{P}\colon\mathcal{A}(\Phi^+(A_{n-1}))\to\mathcal{A}(\Phi^+(A_{n-1}))$ using the formula in \cref{thm:lk} gives an involution satisfying the conditions of~\cref{conj:pan}. Hence, $\mathfrak{P}=\Rvac$ for Type~A. Panyushev~\cite[\S5.1]{panyushev2004adnilpotent} also showed that defining $\mathfrak{P}$ for Type~C by embedding it into the Type~A root poset via $\iota$ and using the Type~A definition of~$\mathfrak{P}$ also gives an involution satisfying~\cref{conj:pan}; so $\mathfrak{P}=\Rvac$ for Type C as well. And Panyushev~\cite[\S5.2]{panyushev2004adnilpotent} showed the same for Type~B; thus $\mathfrak{P}=\Rvac$ again for Type~B. Hence, the only remaining case of~\cref{thm:main_intro} is Type~D. As mentioned, only properties~(i) and~(ii) of~\cref{conj:pan} for $\Rvac$ are in doubt. In fact, since Type~D is simply laced, (ii) is vacuous because all roots in the root system are short roots. Consequently, all we have left to show is that $\#A+\#\Rvac(A)=n$ for all $A\in \mathcal{A}(\Phi^+(D_n))$.

Showing  that $\#A+\#\Rvac(A)=n$ for all $A\in \mathcal{A}(\Phi^+(D_n))$ will take up the remainder of the paper and require quite a lot of work. In particular, we will have to use the bijection of Armstrong--Stump--Thomas~\cite{armstrong2013uniform} between nonnesting and noncrossing partitions, which we discuss in the next section.

\begin{figure}
\begin{center}
\includegraphics[height=4.061cm]{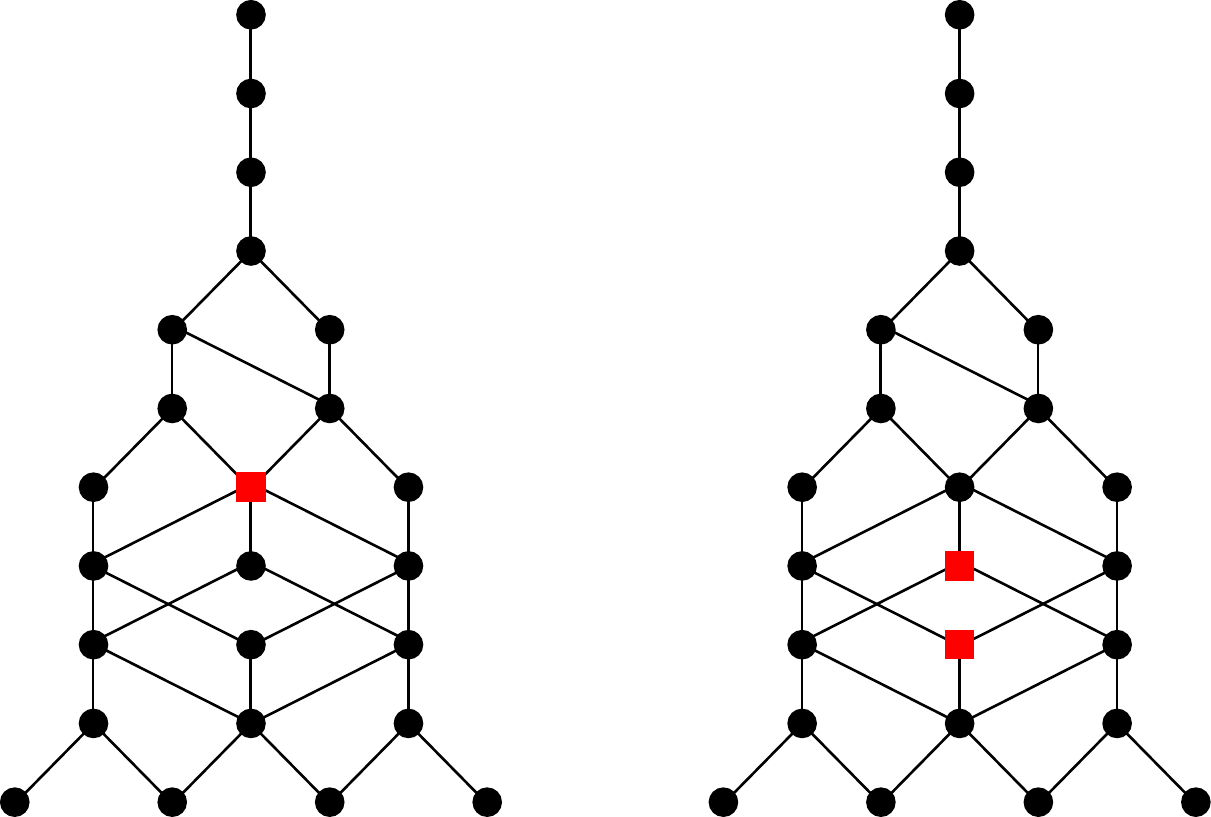}
\end{center}
\caption{An antichain $A\in \mathcal{A}(\Phi^+(F_4))$ (left) and its rowvacuation (right). In this example, $\#A +\#\Rvac(A)\neq 4$.} \label{fig:f4_counterexample}
\end{figure}

\begin{remark} \label{rem:exceptional}
As mentioned in \cref{sec:intro}, the identity $\#A +\#\Rvac(A)=r$ fails to hold when $\Phi$ is one of the exceptional types other than $G_2$. \Cref{fig:f4_counterexample} gives an example of such a failure for $F_4$; there are similar examples for $E_6$, $E_7$, and~$E_8$. It would be interesting to try to modify rowvacuation somehow to obtain Panyushev's desired involution~$\mathfrak{P}$ in the exceptional types.
\end{remark}

\begin{remark}
In~\cite[\S2.8]{hopkins2020birational}, it is explained that if there is a constant $c\in\mathbb{N}$ for which $\#A+\#\Rvac(A)=c$ for every antichain $A\in\mathcal{A}(P)$ of a ranked poset $P$, then necessarily $\frac{1}{\#O}\sum_{A \in O} \#A=\frac{c}{2}$ for every $\row $-orbit $O\subseteq \mathcal{A}(P)$ as well. Hence, \cref{thm:main_intro} implies the second item in \cref{thm:ast_row} for the classical types. However, since we use a lot of the machinery of~\cite{armstrong2013uniform} to prove \cref{thm:main_intro}, this does not really qualify as a new proof of this homomesy result.
\end{remark}

\section{Noncrossing partitions and the AST bijection}\label{sec:Flip}

In this section, we review the $W$-noncrossing partitions and then describe the Armstrong--Stump--Thomas~\cite{armstrong2013uniform} bijection between nonnesting and noncrossing partitions, which interacts very nicely with both rowmotion and rowvacuation.

\subsection{Noncrossing partitions}

We continue to fix a root system $\Phi$ in $V$ with Weyl group~$W\subseteq \GL(V)$. We use $s_i \in \GL(V)$ to denote the \dfn{simple reflection} corresponding to a simple root $\alpha_i$. Recall that $W$ is generated by $\{s_1,\ldots,s_r\}$. A \dfn{Coxeter element}~$c\in W$ is a product of all the simple reflections $s_1,\ldots,s_r$ in some order. From now on, fix a choice of a Coxeter element~$c$. The order of $c$ is, by definition, the Coxeter number~$h$. 

We use $T\subseteq W$ to denote the set of \dfn{reflections}, i.e., all $W$-conjugates of $s_1,\ldots,s_r$. The \dfn{absolute length} of an element $w\in W$, denoted~$\ell_T(w)$, is the minimum length of an expression for $w$ as a product of elements of $T$. The \dfn{absolute order} on $W$ is the partial order where $u \leq w$ if and only if~$\ell_T(w) = \ell_T(u) + \ell_T(u^{-1}w)$. The \dfn{identity element} $e$ is the minimal element of absolute order; the Coxeter elements form a subset of the maximal elements. The poset of \dfn{$W$-noncrossing partitions} $\NC(W,c)$ is defined to be the interval $[e,w] \subseteq W$ in absolute order between the identity element~$e$ and the Coxeter element $c$. For fixed $W$ and varying $c$, the posets $\NC(W,c)$ are isomorphic since all Coxeter elements are conjugate; this is why we just used the notation~$\NC(W)$ for this poset in~\cref{sec:intro}. But from now on, we use the notation $\NC(W,c)$ to emphasize the choice of $c$. 

The $W$-noncrossing partitions were first introduced by Brady and Watt~\cite{brady2002kpi1s} and independently Bessis~\cite{bessis2003dual}, following work of Reiner~\cite{reiner1997noncrossing} in the classical types (see also~\cite{athanasiadis2004noncrossing}). The poset $\NC(W,c)$ is ranked with rank function $\ell_T$, and its rank is $\rk(\NC(W,c))=\ell_T(c)=r$. It is known \cite{BradyWattNoncrossing} that $\NC(W,c)$ is always a lattice. Furthermore, $\NC(W,c)$ is self-dual. In fact, the \dfn{Kreweras complement} $\Krew\colon\NC(W,c)\to\NC(W,c)$ defined by $\Krew(w)\coloneqq cw^{-1}$ is an anti-automorphism of $\NC(W,c)$.

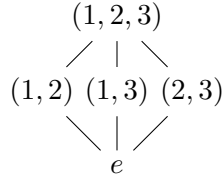
\begin{figure}
\begin{center}
\begin{tikzpicture}
\node (a) at (0,2) {$(1,2,3)$};
\node (b) at (-1,1) {$(1,2)$};
\node (c) at (0,1) {$(1,3)$};
\node (d) at (1,1) {$(2,3)$};
\node (e) at (0,0) {$e$};
\draw (a)--(b)--(e)--(c)--(a)--(d)--(e);
\end{tikzpicture}
\end{center}
\caption{The lattice of noncrossing partitions $\NC(\mathfrak{S}_3)$.} \label{fig:noncrossing_ex}
\end{figure}

Since we will work mostly with the classical types, let us review (some of) their Weyl groups and noncrossing partitions. The Weyl group $W(A_{n-1})$ of Type $A_{n-1}$ is isomorphic to the symmetric group $\mathfrak S_n$. When we view $W(A_{n-1})$ as $\mathfrak S_n$, the simple reflection $s_i$ is the simple transposition~$(i,i+1)$. A standard choice of Coxeter element is $c=s_1s_2\ldots s_{n-1} = (1,2,3,\ldots,n)$ (in cycle notation). \Cref{fig:noncrossing_ex} depicts the lattice of noncrossing partitions in $\mathfrak S_3$ for the standard choice of Coxeter element. Note that $\NC(\mathfrak{S}_n)$ is isomorphic to the classical lattice of noncrossing set partitions of $[n]$, with the Kreweras complement being the classical Kreweras complement of noncrossing set partitions. Meanwhile, the Weyl group $W(D_n)$ of Type~$D_n$ can be viewed as the group of permutations~$w$ of the set $(-[n])\cup[n]$ such that $w(-i)=-w(i)$ for all~$i\in [n]$ and for which the quantity $\#\{i\in[n]\colon w(i)<0\}$ is even. Viewing $W(D_n)$ as a group of permutations of~$(-[n])\cup[n]$, we have
\[ s_i = \begin{cases} (i,i+1)(-i,-i-1) &\textrm{if $1\leq i \leq n-1$}; \\
(n,-n+1)(n-1,-n) &\textrm{if $i=n$}.\end{cases}\]

\subsection{The AST bijection}

The Kreweras complement is usually \emph{not} an involution: it has order $h$ or $2h$. For instance, in Type~A, $\Krew^{2}$ corresponds to rotation of noncrossing set partitions. Hence, one might wonder about the orbit structure of $\Krew$. This is where the connection to Panyushev's work arises. After Panyushev had experimentally exhibited the remarkable properties of rowmotion acting on the root poset in~\cite{panyushev2009orbits}, Bessis and Reiner~\cite{bessis2011cyclic} conjectured that the orbit structure of $\row$ acting on~$\mathcal{A}(\Phi^+)$ is the same as the orbit structure of $\Krew$ acting on~$\NC(W,c)$. Armstrong, Stump, and Thomas~\cite{armstrong2013uniform} proved this conjecture of Bessis--Reiner by defining an explicit bijection between $\mathcal{A}(\Phi^+)$ and $\NC(W,c)$ that equivariantly maps rowmotion to Kreweras complement.

In order to describe their bijection, we need to assume that our Coxeter element is \dfn{bipartite}, i.e., that $c = c_L c_R$, where $L\sqcup R = [r]$ is a bipartition of the nodes of the Dynkin diagram of $\Phi$ and $c_X \coloneqq \prod_{i\in X} s_i$ (the products $c_L$ and $c_R$ are well-defined since the bipartite assumption guarantees that these simple reflections commute). Since the Dynkin diagram of $\Phi$ is always a tree, a bipartite Coxeter element exists. Because all Coxeter elements are conjugate, there is no loss of generality in assuming bipartiteness.

For $J\subseteq [r]$, the \dfn{parabolic subgroup} $W_J$ is the subgroup of $W$ generated by the simple reflections $s_i$ for $i \in J$. If $c_Lc_R$ is a bipartite Coxeter element of $W$, then $c_{L'}c_{R'}$ is a bipartite Coxeter element of $W_J$, where $L' \coloneqq L\cap J$ and $R' \coloneqq R\cap J$; we have a natural inclusion $\NC(W_J,c_{L'}c_{R'})\subseteq \NC(W,c_Lc_R)$. Meanwhile, we use $\Phi^+_J$ to denote the \dfn{parabolic root poset}, which is  $\Phi^+_J \coloneqq \{\alpha \in \Phi^+\colon\alpha \not \geq \alpha_i \textrm{ for all $i\in [r]\setminus J$}\}$.\footnote{We hope the reader can distinguish our notation for the parabolic root poset $\Phi^+_J$ from our notation for the ranks $\Phi^+_i$ of the root poset $\Phi^+$ via context.} For an antichain $A \in \mathcal{A}(\Phi^+)$, we define its \dfn{support} to be 
\[\supp(A) \coloneqq \{i \in [r]\colon \alpha_i \leq \alpha \textrm{ for some $\alpha\in A$}\}.\] 
We can view any antichain $A \in \mathcal{A}(\Phi^+)$ as also an antichain in $\mathcal{A}(\Phi_{\supp(A)}^+)$.

With all this notation in hand, we can now describe the Armstrong--Stump--Thomas nonnesting-to-noncrossing bijection that sends rowmotion to Kreweras complement. As we will see, it is defined inductively, so it is important that we allow the possibility that $\Phi^+$ is reducible, as we have been doing throughout.

\begin{thm}[{Armstrong--Stump--Thomas~\cite{armstrong2013uniform}}] \label{thm:ast}
For a fixed bipartition $L\sqcup R = [r]$ of the Dynkin diagram of $\Phi$, there is a unique bijection $\Theta_{W}\colon \mathcal{A}(\Phi^+)\xrightarrow{\sim} \NC(W,c_Lc_R)$ for which:
\begin{itemize}
\item \emph{(base case)} $\Theta_W(\{\alpha_i\colon i \in L\})=e$;
\item \emph{(equivariance)} $\Theta_W \cdot \row= \Krew\cdot \Theta_W$;
\item \emph{(parabolic induction)} $\Theta_W(A) = \prod_{i \in L\setminus \supp(A)} s_i \cdot \Theta_{W_{\supp(A)}}(A)$ for $A \!\in \mathcal{A}(\Phi^+)$, where $ \Theta_{W_{\supp(A)}}(A) \in \NC(W_{\supp(A)},c_{L\,\cap\, \supp(A)}c_{R\,\cap\, \supp(A)})\subseteq \NC(W,c_Lc_R)$.
\end{itemize}
\end{thm}

\begin{remark}
Armstrong--Stump--Thomas~\cite{armstrong2013uniform} used \cref{thm:ast} to prove~\cref{thm:ast_row}. But also, affirming a conjecture of Bessis and Reiner~\cite{bessis2011cyclic}, they used \cref{thm:ast} to show that $\row$ acting on $\mathcal{A}(\Phi^+)$ satisfies a \dfn{cyclic sieving phenomenon}~\cite{reiner2004cyclic}, where the sieving polynomial is a natural $q$-analogue of~$\Cat(W)$. Recently there has been interest in extending sieving phenomena to dihedral group actions as well~\cite{rao2020dihedral, stier2020dihedral, hopkins2020cyclic}. Hence, it might be interesting to explore sieving phenomenona for $\langle \row, \Rvac \rangle$ acting on $\mathcal{A}(\Phi^+)$.
\end{remark}

\subsection{Flip}

Now that we have stated the Armstrong--Stump--Thomas bijection, we want to bring rowvacuation into the story. In other words, we need an involution on the set of noncrossing partitions of $\Phi$ that plays the role that rowvacuation plays for the nonnesting partitions. Since Kreweras complement is a kind of ``rotation,'' and this hypothetical involution on noncrossing partitions ought to generate a dihedral action with Kreweras complement, we will refer to it as \dfn{flip}. As before, we use the bipartite Coxeter element $c=c_Lc_R$. We define $\Flip\colon \NC(W,c_Lc_R)\to \NC(W,c_Lc_R)$ by $\Flip(w) \coloneqq c_L w^{-1} c_L^{-1}$. Note that since $c_L$ is an involution, $\Flip$ is an involution. It fixes $c$ and permutes the set of reflections, so it is an automorphism of $\NC(W,c_Lc_R)$. It is also easily seen that 
\begin{equation}\label{eqFlipKrew}
\Flip\cdot \Krew=\Krew^{-1}\cdot\Flip.
\end{equation}

\begin{remark}
For non-bipartite $c$, we can define $\Flip\colon\NC(W,c)\to \NC(W,c)$ by conjugating to a bipartite~$c$. In fact, it is not hard to show that the $g\in W$ that conjugate $c$ to $c^{-1}$ form a (left- and right-) $\langle c \rangle$-coset and that they are all necessarily involutions. (The one thing needed to prove this is the well-known fact that the centralizer of $c$ in $W$ is $\langle c \rangle$.) When $h$ is odd, this $\langle c \rangle$-coset is a single $\langle c \rangle$-conjugacy class; when $h$ is even, it consists of two $\langle c \rangle$-conjugacy classes.
\end{remark}

\begin{remark}
Extending the previous remark, we note there are other involutive automorphisms of $\NC(W,c_Lc_R)$ such as the map $\Flip'\colon \NC(W,c_Lc_R)\to \NC(W,c_Lc_R)$ defined by $\Flip'(w) \coloneqq c_R w^{-1} c_R^{-1}$ and others obtained by conjugating $\Flip$ or $\Flip'$ by powers of~$c$. These various involutive automorphisms of $\NC(W,c_Lc_R)$ were previously considered, for instance, by Armstrong in~\cite[\S4.3.4]{armstrong2009generalized}. Our focus on $\Flip$ is ultimately a matter of convention.
\end{remark}

\begin{remark}
That $c$ is conjugate by an involution to its inverse is not a special property of Coxeter elements; Carter~\cite[Theorem~C(iii)]{carter1972conjugacy} proved that \emph{every} element of a Weyl group is conjugate by an involution to its inverse.
\end{remark}

With this definition of~$\Flip$, we can upgrade the equivalence of cyclic actions in \cref{thm:ast} to an equivalence of dihedral actions essentially ``for free,'' using just the general properties of the bijection listed in that theorem. This is what we have asserted in~\cref{thm:ast_intro}:
\[\Theta_W \cdot (\row^{-1}\cdot \Rvac) = \Flip\cdot \Theta_W.\]
The reason that we need to use $\row^{-1}\cdot \Rvac$ instead of just $\Rvac$ in \cref{thm:ast_intro} is because sometimes there are no $A \in \mathcal{A}(\Phi^+)$ fixed by $\Rvac$, while there will always be some fixed points of $\Flip$ (e.g.,~$e$ and~$c$).

In order to prove \cref{thm:ast_intro}, we need to show that $\row^{-1}\cdot \Rvac$ behaves well with respect to parabolic induction. In fact, this is true for any poset $P$. Namely, for an antichain $A \in \mathcal{A}(P)$, slightly abusing notation, let us define its \dfn{support} to be $\supp(A)\coloneqq \{p \in P_0\colon p\leq q \textrm{ for some $q\in A$}\}$. For a subset $X\subseteq P_0$, define $P_{X} \coloneqq \{q \in P\colon q \not\geq p \textrm{ for all $p \in P_0\setminus X$}\}$. We can view any antichain $A \in \mathcal{A}(P)$ as an antichain in $\mathcal{A}(P_{\supp(A)})$.  We let $(\row^{-1}\cdot \Rvac)_{\supp(A)}(A)$ denote the image of~$A$ under $\row^{-1}\cdot \Rvac\colon \mathcal{A}(P_{\supp(A)})\to \mathcal{A}(P_{\supp(A)})$.

\begin{lemma} \label{lem:par_ind}
For any $P$ and $A \in \mathcal{A}(P)$, we have:
\begin{itemize}
\item $\supp((\row^{-1}\cdot \Rvac)(A)) = \supp(A)$;
\item $\row^{-1}\cdot \Rvac(A) = (\row^{-1}\cdot \Rvac)_{\supp(A)}(A)$.
\end{itemize}
\end{lemma}
\begin{proof}
The argument is very similar to the proof of~\cref{prop:rvac_parabolic}. Let us prove the second bulleted item first. We write $\row^{-1}\cdot \Rvac$ as
\[(\boldsymbol{\tau}_{0} \cdot \boldsymbol{\tau}_{1} \cdots \boldsymbol{\tau}_{\rk(P)-1} \cdot \boldsymbol{\tau}_{\rk(P)})  \cdot (\boldsymbol{\tau}_{\rk(P)}) \cdot (\boldsymbol{\tau}_{\rk(P)} \cdot \boldsymbol{\tau}_{\rk(P)-1}) \cdots (\boldsymbol{\tau}_{\rk(P)} \cdot \boldsymbol{\tau}_{\rk(P)-1} \cdots \boldsymbol{\tau}_{1} \cdot \boldsymbol{\tau}_{0}). \]
When we apply the first $\boldsymbol{\tau}_0$ to $A$, it will add to our antichain all of $P_0\setminus \supp(A)$, and until those elements are removed by the $\boldsymbol{\tau}_0$ at the end of this sequence of toggles (which they will be), no element of $P\setminus P_{\supp(A)}$ can be toggled into our antichain. Meanwhile, the status of elements of $P_0\setminus \supp(A)$ has no effect on the toggles~$\tau_p$ for $p\in P_{\supp(A)}$. So indeed, applying the above sequence of toggles to $A$ will have the same effect if we only apply the toggles $\tau_p$ for $p \in P_{\supp(A)}$. This proves the second bulleted item.

It follows from the previous paragraph that $\supp((\row^{-1}\cdot \Rvac)(A)) \subseteq \supp(A)$. But since $(\row^{-1}\cdot \Rvac)^2$ is the identity, we also have the containment $\supp(A)\subseteq \supp((\row^{-1}\cdot \Rvac)(A))$, thus proving the first bulleted item as well.
\end{proof} 

\begin{proof}[Proof of \cref{thm:ast_intro}]
In showing that the bijection $\Theta_{W}\colon \mathcal{A}(\Phi^+)\xrightarrow{\sim} \NC(W,c_Lc_R)$ is uniquely defined by the listed properties, Armstrong, Stump, and Thomas~\cite{armstrong2013uniform} explained that for every $A \in\mathcal{A}(\Phi^+)$, there is a $k \geq 0$ so that $\supp(\row ^{k}(A)) \neq[r]$. Hence, the bijection can be computed inductively as follows: we let $A' \coloneqq \row ^{k}(A)$, where $k\geq 0$ is minimal so that $\supp(A') \neq [r]$; we then compute $w' \coloneqq \Theta_{W_{\supp(A')}}(A')$ inductively; finally, we set $\Theta_{W}(A) \coloneqq \Krew^{-k}(\prod_{i \in L\setminus \supp(A') }s_i \cdot w')$. The base case of this induction is where we use the condition $\Theta_W(\{\alpha_i\colon i \in L\})=e$.\footnote{Actually, we could take as our base case the unique bijection $\Theta_{W}\colon \mathcal{A}(\Phi^+)\xrightarrow{\sim} \NC(W,c)$ for $W$ the trivial group, which shows that the first item in~\cref{thm:ast} is not really needed.}

Thus, to prove this theorem, it suffices to show that:
\begin{itemize}
\item as a base case, $\Theta_W ( (\row^{-1}\cdot \Rvac)(\{\alpha_i\colon i \in L\})) = \Flip(e)$; 
\item if $\Theta_W( (\row^{-1}\cdot \Rvac)(A)) = \Flip(\Theta_W(A))$, then 
\[\Theta_W((\row^{-1}\cdot \Rvac)(\row^{-1}(A))) = \Flip(\Krew^{-1} (\Theta_W(A)));\]
\item if $\Theta_{W_{\supp(A)}}( (\row^{-1}\cdot \Rvac)_{\supp(A)} (A)) = \Flip_{W_{\supp(A)}}(\Theta_{W_{\supp(A)}}(A) )$, then 
\[\Theta_W((\row^{-1}\cdot \Rvac)(A)) = \Flip({\textstyle\prod_{i \in L\setminus \supp (A)}} s_i \cdot \Theta_{W_{\supp(A)}}(A))=\Flip(\Theta_W(A)).\]
\end{itemize}

The first bulleted item is clear since 
\[(\row^{-1}\cdot \Rvac)(\{\alpha_i\colon i \in L\})=\row^{-1}(\{\alpha_i\colon i \in R\}) =\{\alpha_i\colon i \in L\}\] 
and $\Flip(e)=e$.

For the second bulleted item, we use \cref{prop:row_basics} and \eqref{eqFlipKrew} to see that
\begin{align*}
\Theta_W((\row^{-1}\cdot \Rvac) (\row^{-1}(A))) &=  \Theta_W((\row\cdot \row^{-1} \Rvac)(A)) \\
&= \Krew (\Theta_W((\row^{-1}\cdot \Rvac)(A))) \\
&= \Krew (\Flip(\Theta_W(A))) =  \Flip( \Krew^{-1}(\Theta_W(A))). 
\end{align*}

For the third bulleted item, we use \cref{lem:par_ind} and the fact that, by definition, $\Theta_{W_{\supp(A)}}(A) = \prod_{i \in L\setminus \supp(A)} s_i \cdot \Theta_W (A)$ to see that 
\begin{align*}
\Theta_W((\row^{-1}\cdot \Rvac)(A))  &= \Theta_W((\row^{-1}\cdot \Rvac)_{\supp(A)}(A)) \\
&= \prod_{i \in L\setminus \supp(A)} s_i  \cdot \Theta_{W_{\supp(A)}}( (\row^{-1}\cdot \Rvac)_{\supp(A)} (A)) \\
&= \prod_{i \in L\setminus \supp(A)} s_i  \cdot \Flip_{W_{\supp(A)}}(\Theta_{W_{\supp(A)}}(A) ) \\
&= \prod_{i \in L\setminus \supp(A)} s_i  \cdot \Flip_{W_{\supp(A)}} \left(\prod_{i \in L\setminus \supp(A)} s_i \cdot \Theta_W (A) \right)\\
&=\prod_{i \in L\setminus \supp(A)} s_i \cdot c_{L \cap \supp(A)} \cdot \Theta_W (A)^{-1} \cdot \hspace{-15pt} \prod_{i \in L\setminus \supp(A)} \hspace{-10pt} s_i \cdot c_{L \cap \supp(A)} \\
&= c_L (\Theta_W (A)^{-1}) c_L \\
&= \Flip(\Theta_W (A)). &\qedhere
\end{align*}
\end{proof}

Armstrong, Stump, and Thomas proved \cref{thm:ast} by explicitly constructing the map $\Theta_W$ in each of the classical types. In the next two sections, we review their constructions in Types~A and~D. Thanks to \cref{thm:ast_intro}, these explicit constructions will help us understand the effect that $\Rvac$ has on an antichain $A\in\mathcal{A}(\Phi^+(D_n))$.

\section{The AST bijection in Type A}\label{Sec:ASTBijectionA}

In this section, we review the construction of the Armstong--Stump--Thomas bijection in Type~A. It uses noncrossing matchings. 

Let $A=\{[i_1,j_1],\ldots,[i_m,j_m]\} \in \mathcal{A}(\Phi^+(A_{n-1}))$, with $i_1< \cdots < i_m$. Place $2n$ vertices labeled $1^{(0)},2^{(0)},\ldots,n^{(0)},n^{(1)},\ldots,2^{(1)},1^{(1)}$ in clockwise order at evenly-spaced positions around a circle. For $1\leq \ell\leq m$, add a marker on the vertex $j_\ell^{(0)}$ with marking~$i_\ell$. Then, for each $i\in[n]\setminus\{i_1,\ldots,i_m\}$, add a marker on $i^{(1)}$ with marking~$i$. We now add $n$ (straight) edges to the configuration, one at a time. The~$i^\text{th}$ edge that we add connects the vertex with marking $i$ to the closest unmarked vertex that is not already connected to an edge. Here, ``closest'' is determined by moving counterclockwise from the vertex with marking $i$ if that vertex has a $(0)$ superscript and is determined by moving clockwise from that vertex if it has a $(1)$ superscript. The resulting diagram is denoted by $\psi_{A_{n-1}}(A)$. The markings are only used to help with the construction of $\psi_{A_{n-1}}(A)$; we do not consider them to be part of the diagram. 

\begin{figure}[ht]
\begin{center}
\includegraphics[height=3.61cm]{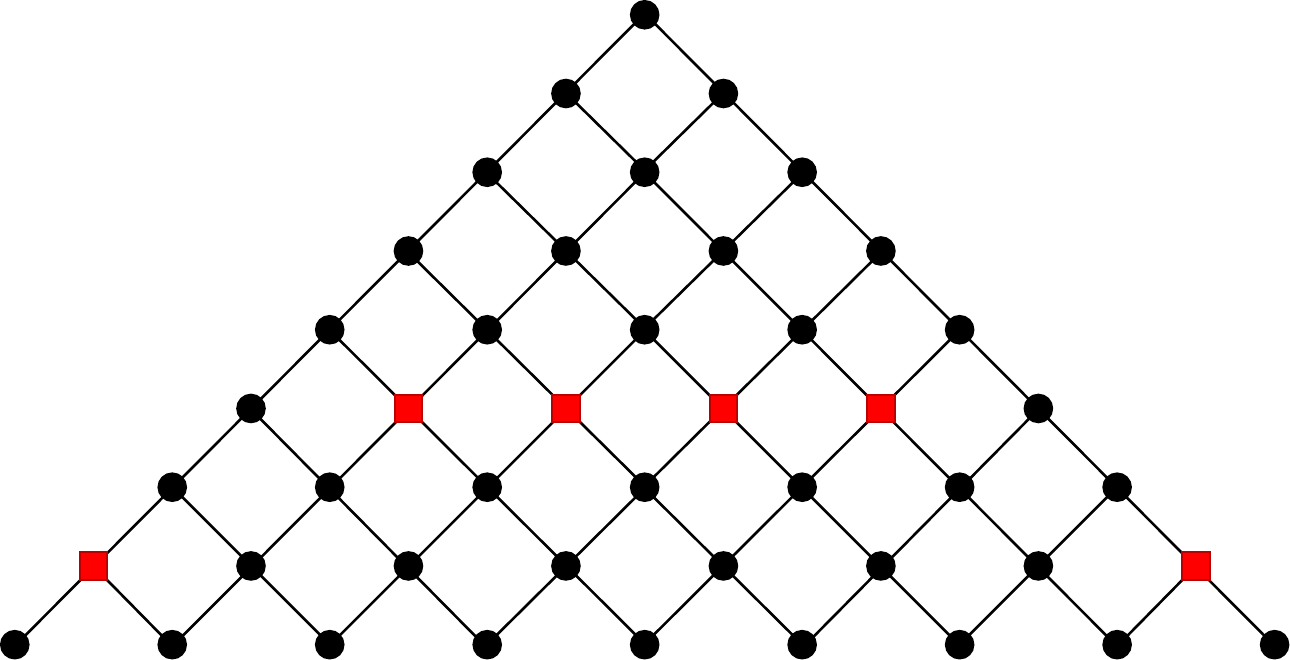}\\ \vspace{.4cm}
\includegraphics[height=5.5cm]{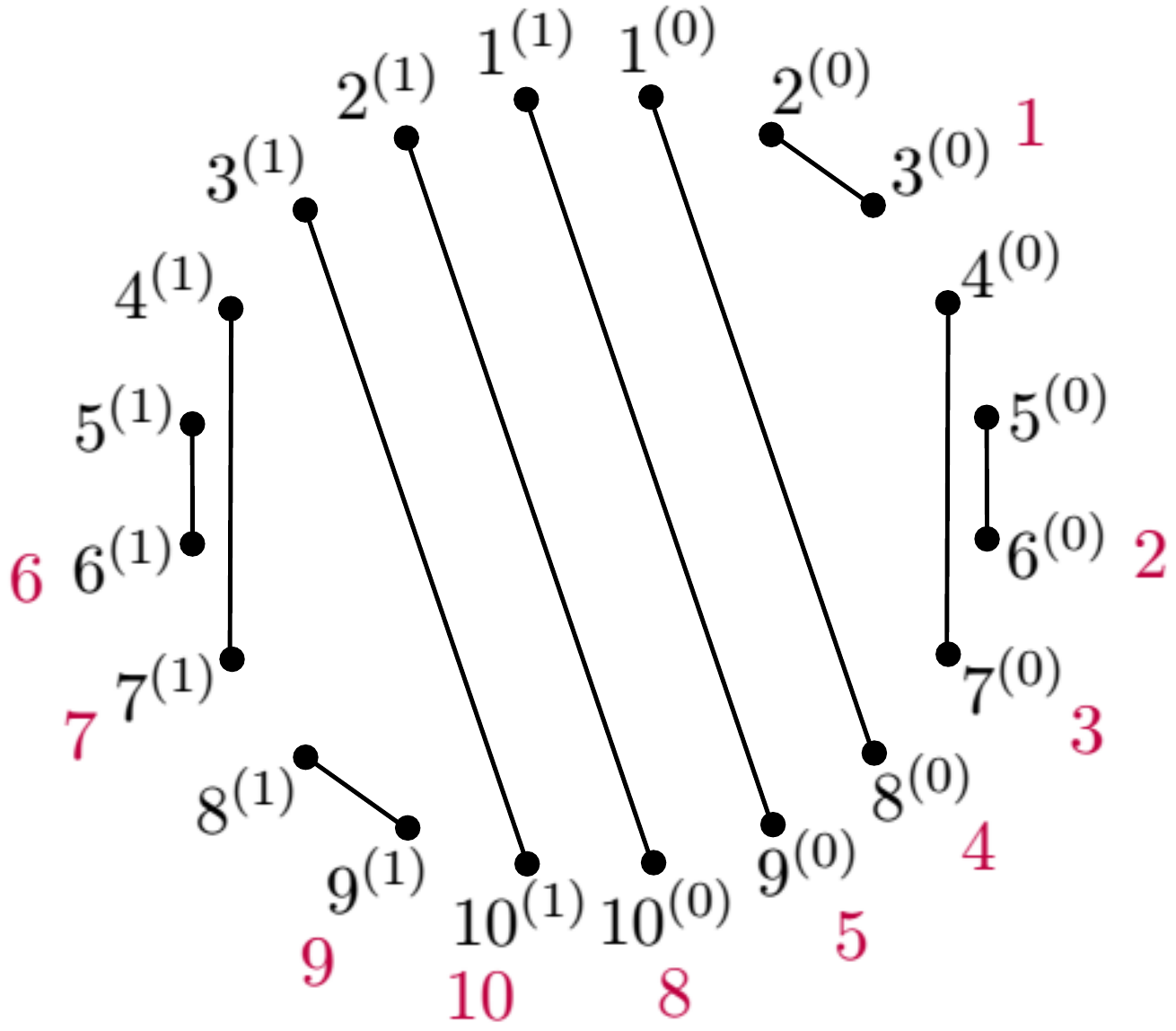}
\end{center}
\caption{An antichain $A\in\mathcal A(\Phi^+(A_9))$ (top) and the corresponding diagram $\psi_{A_9}(A)$ (bottom).}\label{fig:examAfig1}
\end{figure} 

\begin{example}\label{exam:A1}
If $A$ is the antichain $\{[1,3],[2,6],[3,7],[4,8],[5,9],[8,10]\}$ in $\Phi^+(A_9)$ shown on the top of \cref{fig:examAfig1}, then $\psi_{A_9}(A)$ is the diagram shown on the bottom.
The markings (which are not considered to be part of the diagram $\psi_{A_9}(A)$) are represented by purple numbers. For example, the fact that $[1,3]\in A$ tells us that the vertex $3^{(0)}$ should be given the marking ${\color{purple}1}$. Once the markings are placed, the edges are drawn as prescribed above.
\end{example}

The next step is to relabel the vertices in $\psi_{A_{n-1}}(A)$ according to a bipartite Coxeter element $c=c_Lc_R\in W(A_{n-1})$. We choose $L\sqcup R$ to be the unique bipartition of the Dynkin diagram of $\Phi(A_{n-1})$ such that $1\in L$. This means that $c_L=s_1s_3\cdots s_{n-2}$ if $n$ is odd and $c_L=s_1s_3\cdots s_{n-1}$ if $n$ is even. Let $p_1,\ldots,p_n$ be the sequence obtained by listing the even elements of $[n]$ in increasing order and then listing the odd elements of $[n]$ in decreasing order. The cycle decomposition of $c$ (viewed as a permutation in $\mathfrak S_n$) is $(p_1,\ldots,p_n)$. In the diagram $\psi_{A_{n-1}}(A)$, relabel the vertices $1^{(0)},2^{(0)},\ldots,n^{(0)},n^{(1)},\ldots,2^{(1)},1^{(1)}$ as $p_1^{(0)},p_1^{(1)},p_2^{(0)},p_2^{(1)},\ldots,p_n^{(0)},p_n^{(1)}$, respectively. Let us call the relabeled diagram $\varphi_{A_{n-1}}(A)$. 

In \cite{armstrong2013uniform}, it is shown that $\varphi_{A_{n-1}}(A)$ is a noncrossing matching, meaning that each vertex is incident to exactly one edge and that no two edges cross each other. Furthermore, each edge has one endpoint with a $(0)$ superscript and one endpoint with a $(1)$ superscript. This allows us to define a permutation $w\in \mathfrak S_n$ by declaring $w(i)=j$ whenever $i^{(1)}$ is connected to $j^{(0)}$ via an edge in $\varphi_{A_{n-1}}(A)$. Armstrong, Stump, and Thomas proved that this definition yields the bijection $\Theta_{W(A_{n-1})}$ from \cref{thm:ast}. 

\begin{thm}[{\cite[\S3.1]{armstrong2013uniform}}]\label{thm:explicitastA}
Let $\Theta_{W(A_{n-1})}\colon \mathcal A(\Phi^+(A_{n-1}))\to\NC(W(A_{n-1}),c)$ be the bijection from \cref{thm:ast}, where $c=c_Lc_R=(p_1,\ldots,p_n)$ is the bipartite Coxeter element defined above. For every $A\in\mathcal A(\Phi^+(A_{n-1}))$, we have $\Theta_{W(A_{n-1})}(A)(i)=j$ if and only if $i^{(1)}$ is connected to $j^{(0)}$ via an edge in $\varphi_{A_{n-1}}(A)$. 
\end{thm}

\begin{figure}[ht]
\begin{center}
\includegraphics[height=5.5cm]{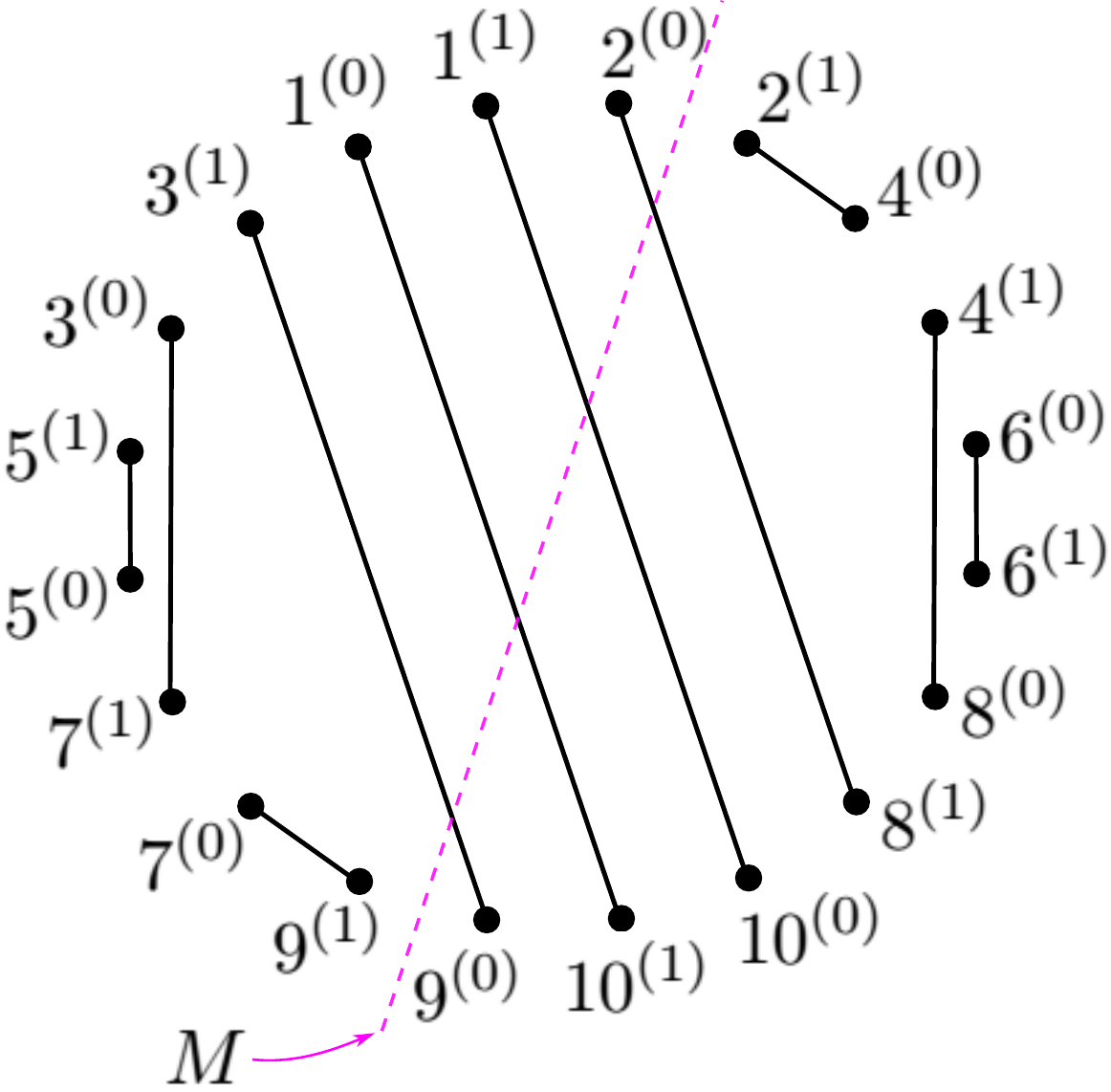}\hspace{1.28cm} \includegraphics[height=5.5cm]{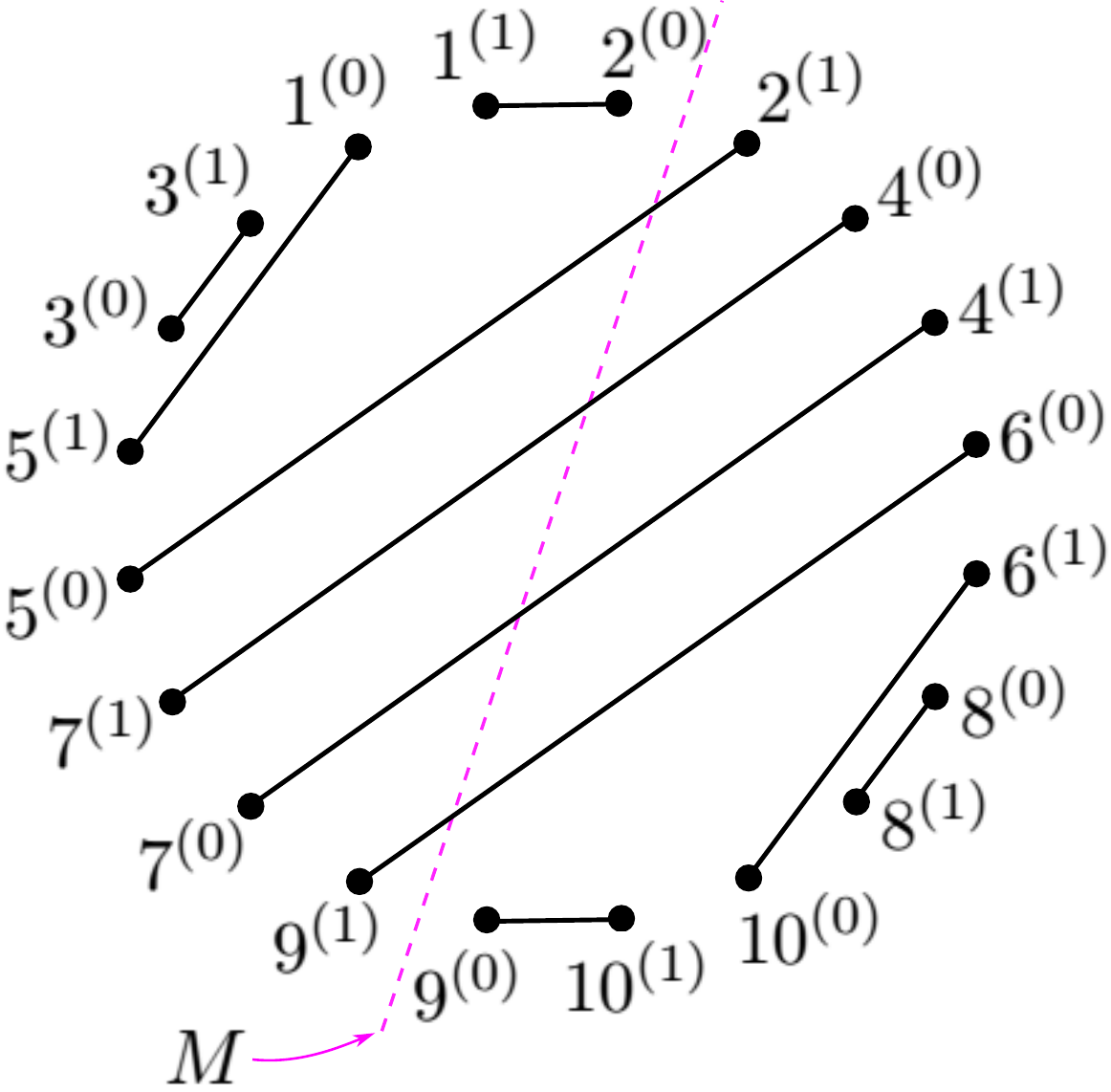}
\end{center}
\caption{The diagrams $\varphi_{A_{9}}(A)$ (left) and $\varphi_{A_9}(\row\cdot\Rvac(A))$ (right), where $A$ is the antichain from \cref{exam:A1}. We have also drawn the line $M$ in each diagram. Note that $\varphi_{A_9}(\row\cdot\Rvac(A))$ is obtained from $\varphi_{A_9}(A)$ by reflecting the edges through $M$.}\label{fig:examfigA1}
\end{figure}

\begin{example}\label{exam:A2}
Let $A$ be as in \cref{exam:A1}. By relabeling the vertices of $\psi_{A_{9}}(A)$ in the manner described above, we obtain the diagram $\varphi_{A_{9}}(A)$ shown on the left in \cref{fig:examfigA1}. \Cref{thm:explicitastA} tells us that $\Theta_{W(A_9)}(A)\in\NC(W(A_9),c)$ is the permutation in $\mathfrak S_{10}$ with cycle decomposition $(1,10)(2,4,8)(3,9,7)(5)(6)$. 
\end{example}

To end this section, we find a simple relationship between the diagrams $\varphi_{A_{n-1}}(A)$ and $\varphi_{A_{n-1}}(\row\cdot \Rvac(A))$ that we will need later. By combining \cref{prop:row_basics} with \cref{thm:ast,thm:ast_intro}, we find that \[\Theta_{W(A_{n-1})}\cdot\row\cdot\Rvac=\Theta_{W(A_{n-1})}\cdot\row^2\cdot\row^{-1}\cdot\Rvac=\Krew^2\cdot\Flip\cdot\Theta_{W(A_{n-1})},\] so 
\begin{align*}
\Theta_{W(A_{n-1})}(\row\cdot \Rvac(A))&=\Krew^2\cdot \Flip(\Theta_{W(A_{n-1})}(A))=c\Flip(\Theta_{W(A_{n-1})}(A))c^{-1}\\
&=c_Lc_Rc_L\Theta_{W(A_{n-1})}(A)^{-1}c_Lc_Rc_L.
\end{align*}
In each of the diagrams $\varphi_{A_{n-1}}(A)$ and $\varphi_{A_{n-1}}(\row \cdot \Rvac(A))$, let $M$ be the line through the center of the circle that is equidistant from $2^{(0)}$ and $2^{(1)}$, as shown in \cref{fig:examfigA1}. Let~$\Omega$ be the reflection of the plane through the line $M$. We think of $\Omega$ as acting on the diagrams $\varphi_{A_{n-1}}(A)$ and $\varphi_{A_{n-1}}(\row\cdot\Rvac(A))$. For each $i\in[n]$, we readily compute that $\Omega(i^{(0)})=(c_Lc_Rc_L(i))^{(1)}$ and $\Omega(i^{(1)})=(c_Lc_Rc_L(i))^{(0)}$. With these observations in hand, the next lemma follows immediately from \cref{thm:explicitastA}.
\cref{fig:examfigA1} illustrates this lemma. 

\begin{lemma}\label{lem:ReflectionA}
Let $A\in\mathcal A(\Phi^+(A))$, and preserve the notation from above. The diagram $\varphi_{A_{n-1}}(\row\cdot \Rvac(A))$ is obtained from $\varphi_{A_{n-1}}(A)$ by reflecting all of the edges through the line $M$ (and leaving the vertices unchanged).
\end{lemma} 

\section{The AST Bijection in Type D}\label{Sec:ASTBijectionD}

In this section, we review the pieces of the Armstrong--Stump--Thomas bijection in Type~D that we will need later. We then use those pieces to prove a lemma that will be crucial in the next section. 

Recall from~\cref{subsec:RootPosets} that we have an automorphism $\delta\colon \Phi^+(D_n)\to\Phi^+(D_n)$, a natural quotient map $\gamma\colon \Phi^+(D_n)\to\Phi^+(C_{n-1})$ induced by $\delta$, and a natural inclusion $\iota\colon \mathcal P(\Phi^+(C_{n-1}))\to\mathcal P(\Phi^+(A_{2n-3}))$ obtained by ``unfolding.'' 

The antichains $A$ of $\Phi^+(D_n)$ with $\delta(A)=A$ descend to antichains of $\Phi^+(C_{n-1})$ and so are easily dealt with. Thus, in this section, we fix an $A \in \mathcal{A}(\Phi^+(D_n))$ for which $\delta(A)\neq A$. Then $\iota(\gamma(A))$ is a subset of $\Phi^+(A_{2n-3})$ that is symmetric about the central vertical axis. Note that $\iota(\gamma(A))$ is not necessarily an antichain. As before, we identify the elements of $\Phi^+(A_{2n-3})$ with intervals $[i,j]$ for~$1\leq i<j\leq 2n-2$. Let $\mathcal{Q}$ be the set of elements $[i,j]\in\Phi^+(A_{2n-3})$ such that $i\leq n-1$ and $j\geq n$. Equivalently, $\mathcal Q$ consists of the elements greater than or equal to the minimal element $[n-1,n]$. Let~$[i_1,j_1],\ldots,[i_k,j_k]$ be the elements of~$\iota(\gamma(A))\cap \mathcal{Q}$ written in lexicographic order. (Note that $\delta(A)\neq A$ implies $k\geq 1$.) Let $\widehat A$ be the set of elements of $\Phi^+(A_{2n-3})$ obtained from $\iota(\gamma(A))$ by removing the elements of $\iota(\gamma(A))\cap \mathcal{Q}$ and replacing them with $[i_1,j_2],[i_2,j_3],\ldots,[i_{k-1},j_k]$. Lemma~3.9 in~\cite{armstrong2013uniform} states that $\widehat A$ is an antichain of~$\Phi^+(A_{2n-3})$ that is symmetric about the central vertical axis, i.e., with $\eta(\widehat{A})=\widehat{A}$. 

\begin{example}\label{exam:astconstruction1}
If $A$ is the antichain of $\Phi^+(D_6)$ shown in red in the top left of \cref{fig:examfig1}, then $\iota(\gamma(A))$ is the subset of $\Phi^+(A_9)$ shown in red in the top right. The set~$\mathcal{Q}$ consists of the elements of $\Phi^+(A_9)$ lying inside the green square in the figure. The elements of $\iota(\gamma(A))\cap \mathcal{Q}$ are $[2,6]$, $[3,6]$, $[4,7]$, $[5,8]$, and $[5,9]$; $\widehat A\in\mathcal A(\Phi^+(A_9))$ is obtained from $\iota(\gamma(A))$ by replacing these $5$ elements with $[2,6]$, $[3,7]$, $[4,8]$, and $[5,9]$. Notice that $\widehat A$, which is depicted in red in the bottom of \cref{fig:examfig1}, is indeed an antichain of $\Phi^+(A_9)$. In fact, $\widehat A$ is the same as the antichain from \cref{fig:examAfig1}.
\end{example}

\begin{figure}[ht]
\begin{center}
\includegraphics[height=3.61cm]{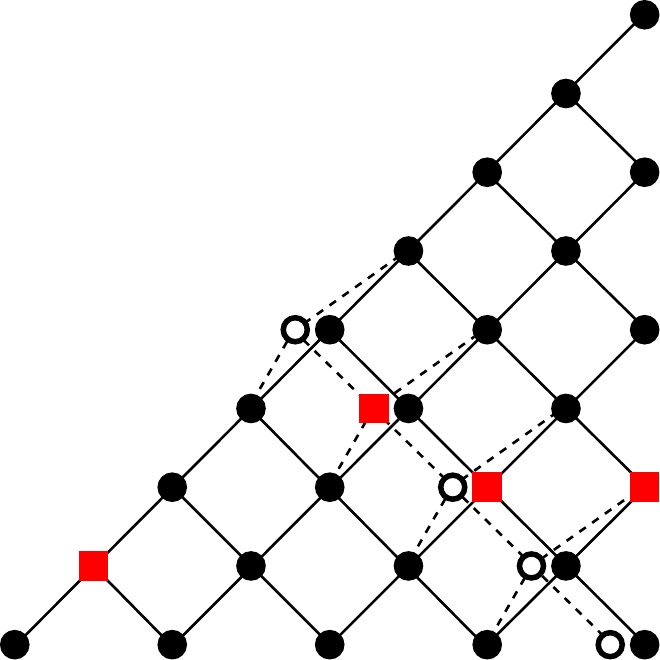}\hspace{1.7cm} \includegraphics[height=3.791cm]{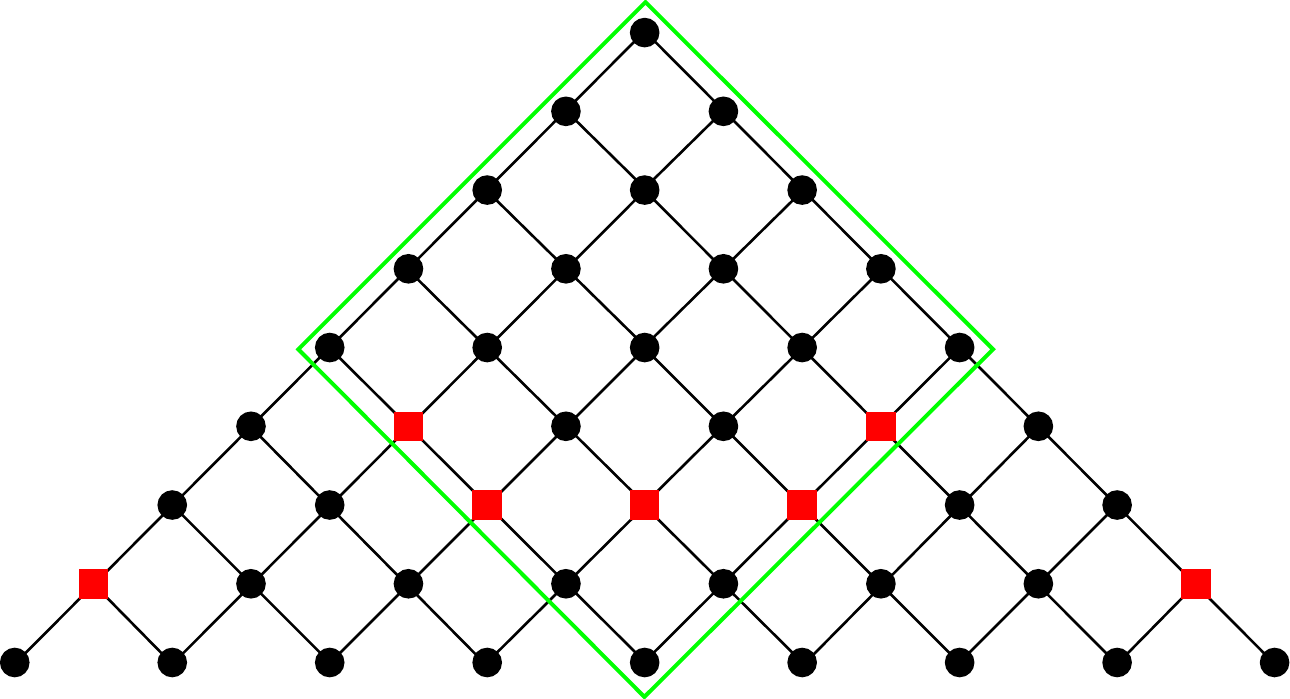}\\ \vspace{.3cm} \includegraphics[height=3.791cm]{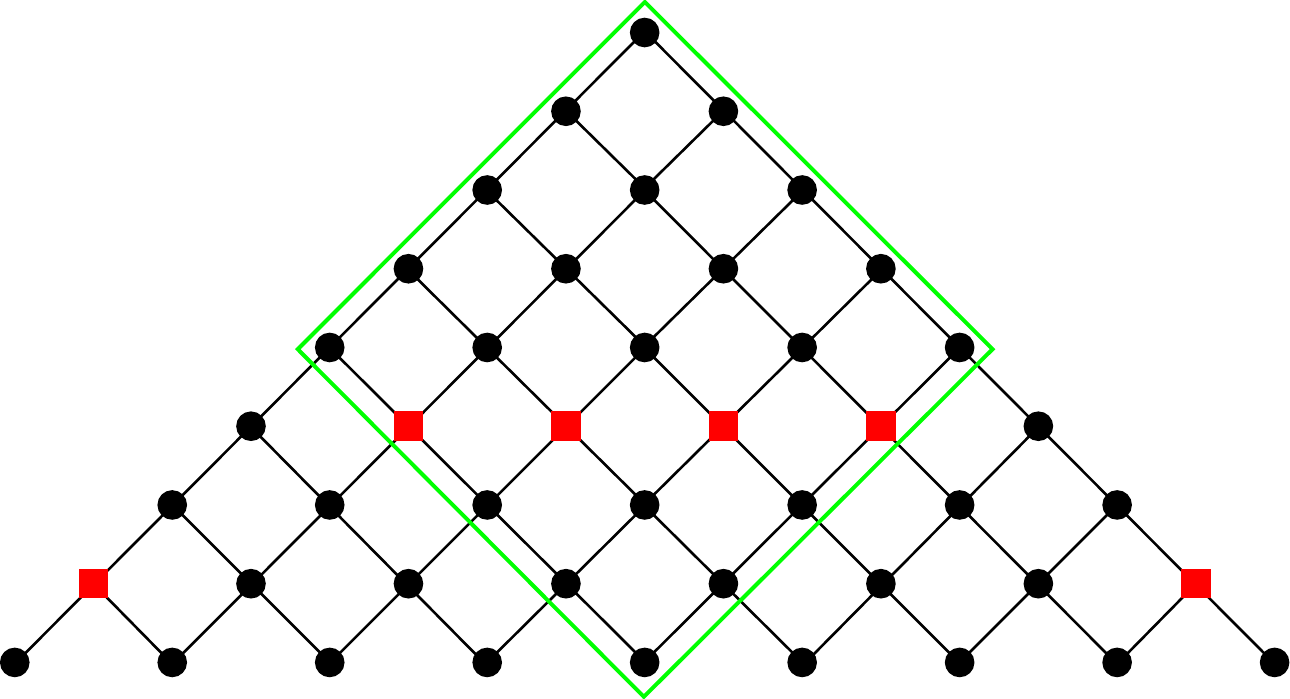} 
\end{center}
\caption{From the antichain $A\in\mathcal A(\Phi^+(D_6))$ (top left), we obtain the set $\iota(\gamma(A))\subseteq\Phi^+(A_9)$ (top right) and the antichain $\widehat A\in\mathcal A(\Phi^+(A_9))$ (bottom). Note that $\iota(\gamma(A))$ is not an antichain.}\label{fig:examfig1}
\end{figure}

Since $\widehat A$ is an antichain of $\Phi^+(A_{2n-3})$, we can consider the diagram $\varphi_{A_{2n-3}}(\widehat A)$ defined in \cref{Sec:ASTBijectionA}. Recall that $M$ is the line through the center of the circle that is equidistant from $2^{(0)}$ and $2^{(1)}$. Let $M^\perp$ denote the line through the center of the circle that is perpendicular to $M$. Let $H$ be the set of vertices in the diagram that are on the same side of $M^\perp$ as $2^{(0)}$, and let $\overline H$ be the set of vertices in the diagram that are on the opposite side of $M^\perp$. We say an edge in the diagram is \dfn{transverse} if it has one endpoint in $H$ and one endpoint in $\overline H$. Lemma~3.11 in \cite{armstrong2013uniform} implies that there are at least two transverse edges in $\varphi_{A_{2n-3}}(\widehat A)$; by removing the two transverse edges that are closest to the center of the circle, we obtain a new diagram that we denote by $\xi_{A_{2n-3}}(\widehat A)$. 

We now relabel the vertices in $\xi_{A_{2n-3}}(\widehat A)$ according to a bipartite Coxeter element $c=c_Lc_R\in W(D_n)$. We choose $L\sqcup R$ to be the unique bipartition of the Dynkin diagram of $\Phi(D_n)$ such that $1\in L$. This means that $c_L=s_1s_3\cdots s_{n-4}s_{n-2}$ if $n$ is odd and $c_L=s_1s_3\cdots s_{n-3}s_{n-1}s_n$ if $n$ is even. Let $q_1,\ldots,q_{2n-2}$ be the sequence of numbers obtained by listing the even elements of $[n-1]$ in increasing order, then listing the odd elements of $-[n-1]$ in increasing order, then listing the even elements of $-[n-1]$ in decreasing order, and finally listing the odd elements of $[n-1]$ in decreasing order. The cycle decomposition of $c$ is $(q_1,\ldots,q_{2n-2})(n,-n)$. Relabel the vertices in the diagram $\xi_{A_{2n-3}}(\widehat A)$, starting at $2^{(0)}$ and moving clockwise, as \[q_1^{(0)},q_1^{(1)},q_2^{(0)},q_2^{(1)},\ldots,q_{2n-2}^{(0)},q_{2n-2}^{(1)}.\] Let us call the relabeled diagram $\varphi_{D_n}(A)$. 

\begin{example}\label{exam:astconstruction2}
Let $A$ and $\widehat A$ be as in \cref{exam:astconstruction1}. The diagram $\varphi_{A_9}(\widehat A)$, which we computed in \cref{exam:A2}, is drawn again on the left in \cref{fig:examfig3}. We have now included the line $M^\perp$ in the figure. Furthermore, we have colored the vertices in $H$ red and colored the vertices in $\overline H$ blue. The two transverse edges that are closest to the center of the circle are the one connecting $1^{(1)}$ to $10^{(0)}$ and the one connecting $10^{(1)}$ to $1^{(0)}$; these two edges are removed to produce $\xi_{A_9}(\widehat A)$. After relabeling the vertices of $\xi_{A_9}(\widehat A)$ in the manner described above, we obtain the diagram $\varphi_{D_6}(A)$ shown on the right in \cref{fig:examfig3}.
\end{example}

\begin{figure}[ht]
\begin{center}
\includegraphics[height=5.5cm]{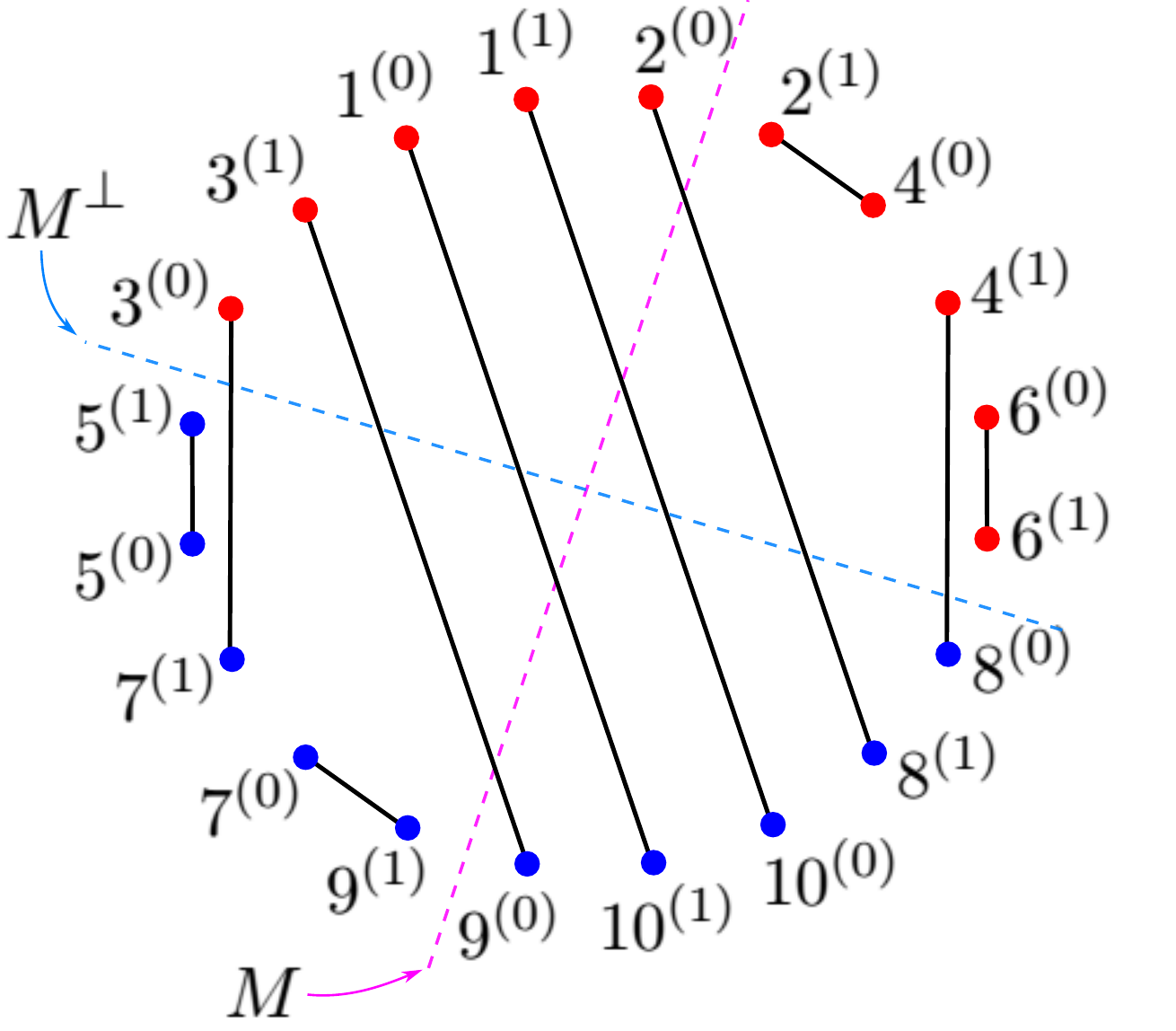}\hspace{1.28cm} \includegraphics[height=5.5cm]{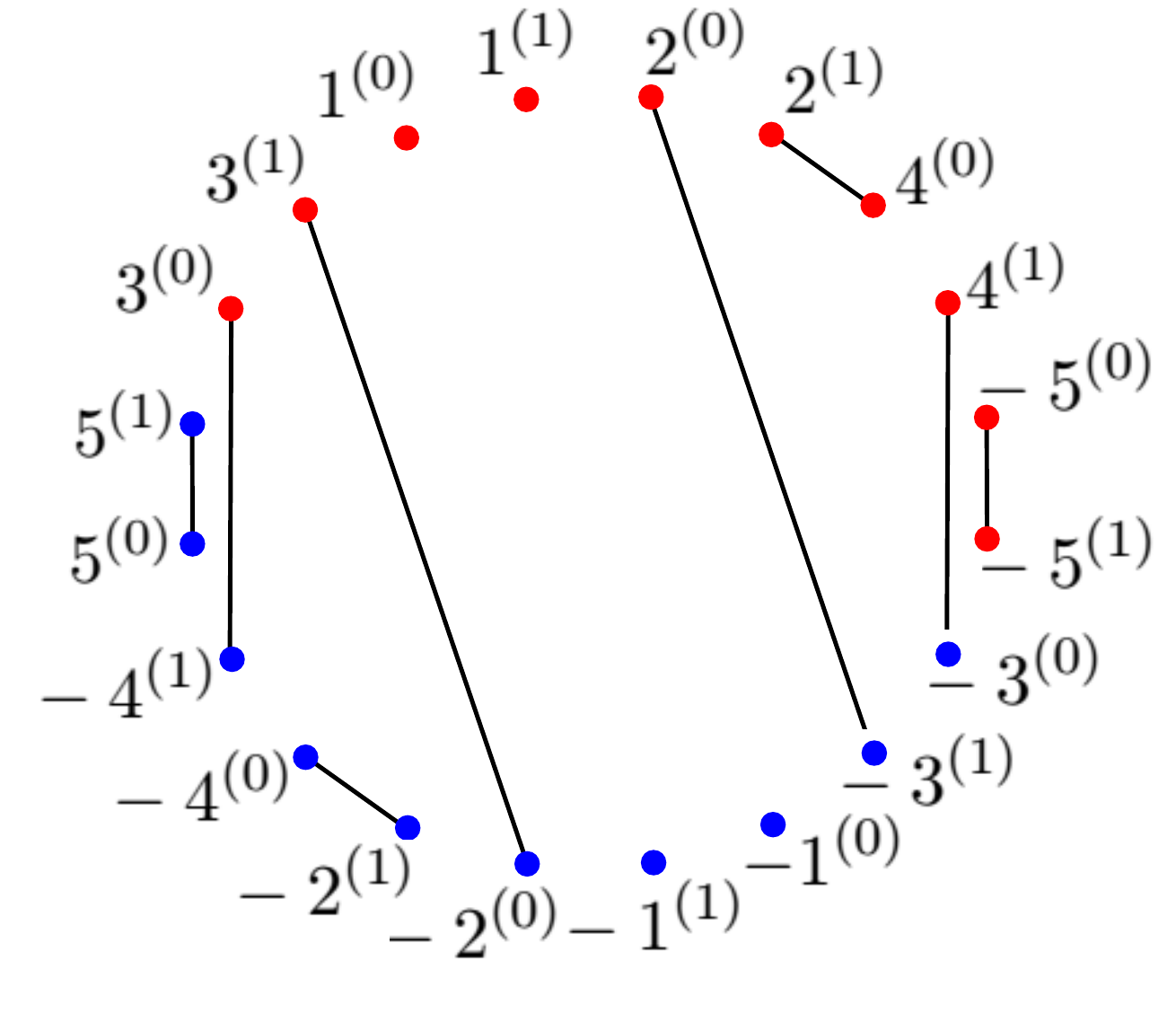}
\end{center}
\caption{The diagrams $\varphi_{A_9}(\widehat A)$ (left) and $\varphi_{D_6}(A)$ (right) from \cref{exam:astconstruction2}.}\label{fig:examfig3}
\end{figure}

Armstrong, Stump, and Thomas used the diagrams we have constructed in this section to give an explicit description of the map $\Theta_{W(D_n)}$. We will only need part of their description, which we state in the following theorem. Recall that we view~$W(D_n)$ as the group of permutations $w$ of $(-[n])\cup[n]$ such that $w(-i)=-w(i)$ for all $i\in[n]$ and such that $\#\{i\in[n]:w(i)<0\}$ is even. 

\begin{thm}[{\cite[\S3.3]{armstrong2013uniform}}]\label{thm:astexplicitD}
Let $\Theta_{W(D_n)}\colon \mathcal A(\Phi^+(D_n))\to\NC(W(D_n),c)$ be the bijection from \cref{thm:ast}, where $c=c_Lc_R=(q_1,\ldots,q_{2n-2})(n,-n)$ is the bipartite Coxeter element defined above.  Let $A\in\mathcal A(\Phi^+(D_n))$ be such that $\delta(A)\neq A$. The diagram $\varphi_{D_n}(A)$ defined above is left unchanged if we rotate each of the edges by~$180^\circ$ about the center of the circle. If there is an edge in $\varphi_{D_n}(A)$ with endpoints $i^{(1)}$ and~$j^{(0)}$, then $\Theta_{W(D_n)}(A)(i)=j$. There exist $x_A,y_A\in[n-1]$ such that the four vertices of $\varphi_{D_n}(A)$ not incident to any edges are $x_A^{(0)},-x_A^{(0)},y_A^{(1)},-y_A^{(1)}$ and such that $\Theta_{W(D_n)}(A)(\{n,-n\})=\{x_A,-x_A\}$ and $\Theta_{W(D_n)}(A)(\{y_A,-y_A\})=\{n,-n\}$. 
\end{thm}

\begin{example}\label{exam:astconstruction3}
Let $A$ be as in \cref{exam:astconstruction1} and \cref{exam:astconstruction2}. The diagram $\varphi_{D_6}(A)$ (shown on the right in \cref{fig:examfig3}) is clearly left unchanged when its edges are rotated by $180^\circ$ about the center of the circle. \cref{thm:astexplicitD} tells us that most of the values of $\Theta_{W(D_6)}(A)$ are determined by $\varphi_{D_6}(A)$. For example, the edge between $-3^{(1)}$ and $2^{(0)}$ tells us that $\Theta_{W(D_6)}(A)(-3)=2$. In this example, $x_A=y_A=1$, so \cref{thm:astexplicitD} tells us that $\Theta_{W(D_6)}(A)(\{1,-1\})=\{6,-6\}$ and $\Theta_{W(D_6)}(A)(\{6,-6\})=\{1,-1\}$.
\end{example}

We are now in a position to prove the main lemma of this section, which states that $\row \cdot \Rvac(\widehat A)=\widehat{\row\cdot \Rvac(A)}$. When we compute $\row\cdot \Rvac(\widehat A)$, we are viewing $\row$ and $\Rvac$ as operators on $\mathcal A(\Phi^+(A_{2n-3}))$; when we compute $\widehat{\row\cdot \Rvac(A)}$, we are viewing $\row$ and $\Rvac$ as operators on $\mathcal A(\Phi^+(D_n))$. Note that for any antichain $B \in \Phi^+(D_n)$ obtained from $A$ via a series of rowmotions and rowvacuations, we still have $\delta(B)\neq B$ because rowvacuation and rowmotion commute with poset automorphisms like~$\delta$. Hence, it makes sense to speak of $\widehat{B}$.

\begin{lemma}\label{lem:HatCommutes}
If $A\in\mathcal A(\Phi^+(D_n))$ is such that $\delta(A)\neq A$, then 
\[\row\cdot \Rvac(\widehat A)=\widehat{\row\cdot \Rvac(A)}.\]
\end{lemma}
\begin{proof}
To ease notation, let $B=\row\cdot \Rvac(A)$. By combining \cref{prop:row_basics} with \cref{thm:ast,thm:ast_intro}, we find that 
\[\Theta_{W(D_n)}\cdot\row\cdot\Rvac=\Theta_{W(D_n)}\cdot\row^2\cdot\row^{-1}\cdot\Rvac=\Krew^2\cdot\Flip\cdot\Theta_{W(D_n)},\] 
so  
\begin{align*} 
\Theta_{W(D_n)}(B)&=\Krew^2\cdot \Flip(\Theta_{W(D_n)}(A))=c\Flip(\Theta_{W(D_n)}(A))c^{-1}\\
&=c_Lc_Rc_L\Theta_{W(D_n)}(A)^{-1}c_Lc_Rc_L,
\end{align*}
where $c=c_Lc_R\in W(D_n)$ is the bipartite Coxeter element defined above. In each of the diagrams $\varphi_{D_n}(A)$ and $\varphi_{D_n}(B)$, let $M$ be the line through the center of the circle that is equidistant from $2^{(0)}$ and $2^{(1)}$ (this is the same as the line $M$ shown on the left in \cref{fig:examfig3}). Let $\Omega$ be the reflection of the plane through the line $M$. Notice that $\Omega(H)=H$ and $\Omega(\overline H)=\overline H$. We think of $\Omega$ as acting on~$\varphi_{D_n}(A)$ and $\varphi_{D_n}(B)$. For each $i\in(-[n-1])\cup[n-1]$, we readily compute that $\Omega(i^{(0)})=(c_Lc_Rc_L(i))^{(1)}$ and~$\Omega(i^{(1)})=(c_Lc_Rc_L(i))^{(0)}$. Combining these observations with \cref{thm:astexplicitD}, we find that $\varphi_{D_n}(B)$ is obtained from $\varphi_{D_n}(A)$ by reflecting all of the edges through~$M$ (and leaving the vertices unchanged). Note that we are heavily using the fact that~$\Omega$ preserves $H$ and $\overline H$; indeed, this guarantees that if $E$ and $E'$ are the transverse edges that are removed from $\varphi_{A_{2n-3}}(\widehat A)$ in the construction of~$\varphi_{D_n}(A)$, then $\Omega(E)$ and $\Omega(E')$ are still transverse. These are precisely the edges that are removed from $\varphi_{A_{2n-3}}(\widehat B)$ in the construction of $\varphi_{D_n}(B)$.  

We now know that $\varphi_{D_n}(B)$ is obtained from $\varphi_{D_n}(A)$ by reflecting all of the edges through $M$. It follows that $\varphi_{A_{2n-3}}(\widehat B)$ is obtained from $\varphi_{A_{2n-3}}(\widehat A)$ by reflecting all of the edges through $M$ (we are using the fact that $\varphi_{A_{2n-3}}(\widehat A)$ and $\varphi_{A_{2n-3}}(\widehat B)$ are noncrossing to see that they can be reconstructed uniquely from $\xi_{A_{2n-3}}(\widehat A)$ and $\xi_{A_{2n-3}}(\widehat B)$). We also know by \cref{lem:ReflectionA} that $\varphi_{A_{2n-3}}(\row\cdot \Rvac(\widehat A))$ is obtained from $\varphi_{A_{2n-3}}(\widehat A)$ by reflecting all of the edges through $M$. This means that $\varphi_{A_{2n-3}}(\widehat B)=\varphi_{A_{2n-3}}(\row\cdot \Rvac(\widehat A))$, so it follows from \cref{thm:explicitastA} that $\Theta_{W(A_{2n-3})}(\widehat B)=\Theta_{W(A_{2n-3})}(\row\cdot \Rvac(\widehat A))$. Because $\Theta_{W(A_{2n-3})}$ is a bijection, we must have $\widehat B=\row\cdot \Rvac(\widehat A)$, as desired. 
\end{proof}

\begin{figure}[ht]
\begin{center}
\includegraphics[height=3.61cm]{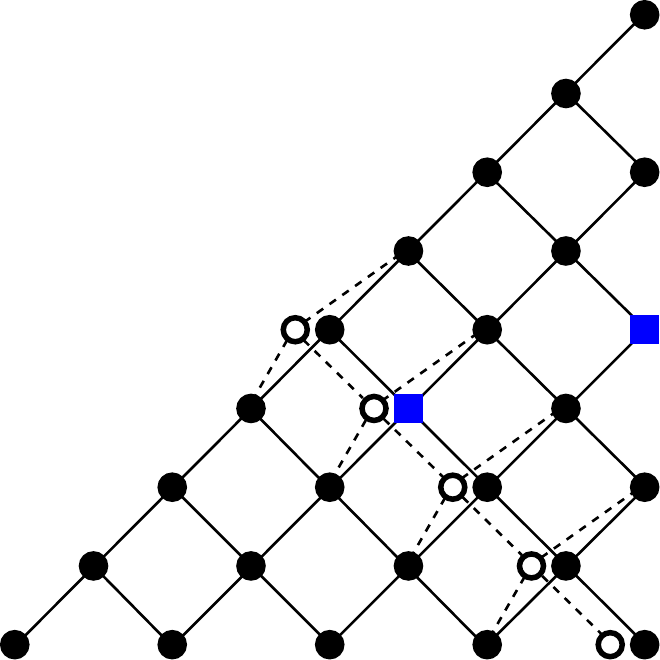}\hspace{1.7cm} \includegraphics[height=3.61cm]{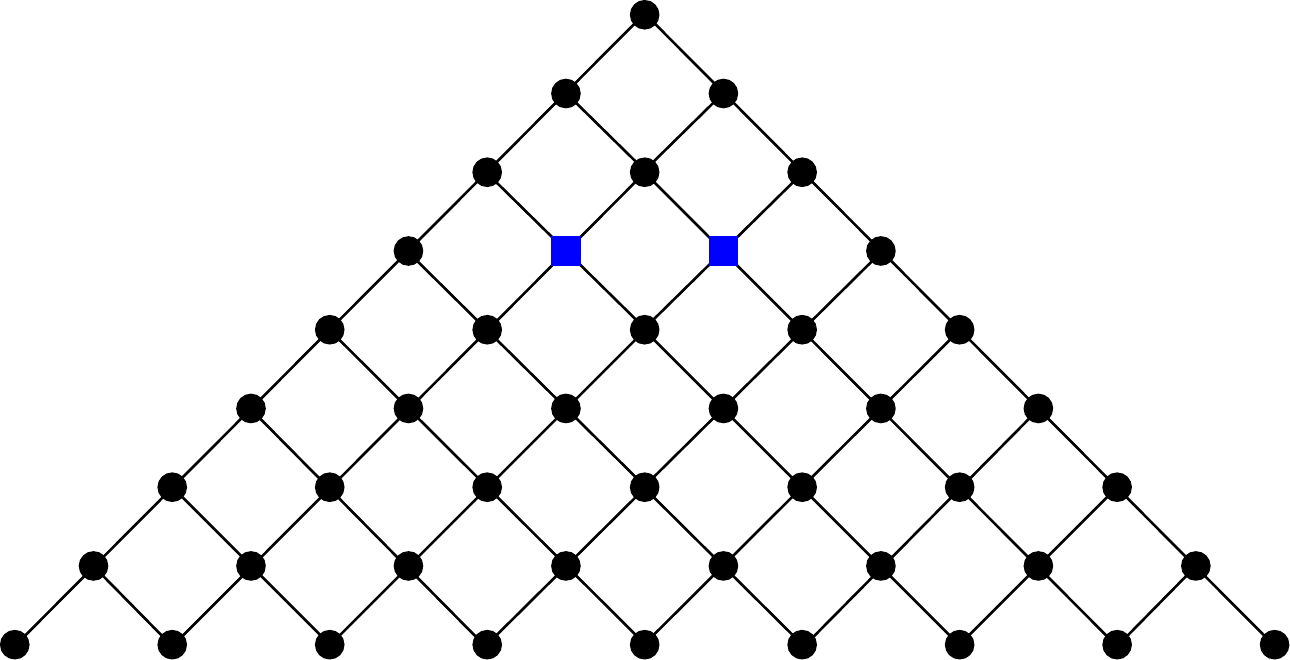} 
\end{center}
\caption{The antichains $\row\cdot \Rvac(A)$ (left) and $\row\cdot \Rvac(\widehat A)$ (right), where $A$ and $\widehat A$ are the antichains from \cref{exam:astconstruction1}.}\label{fig:finalexam}
\end{figure}

\begin{figure}[ht]
\begin{center}
\includegraphics[height=5.5cm]{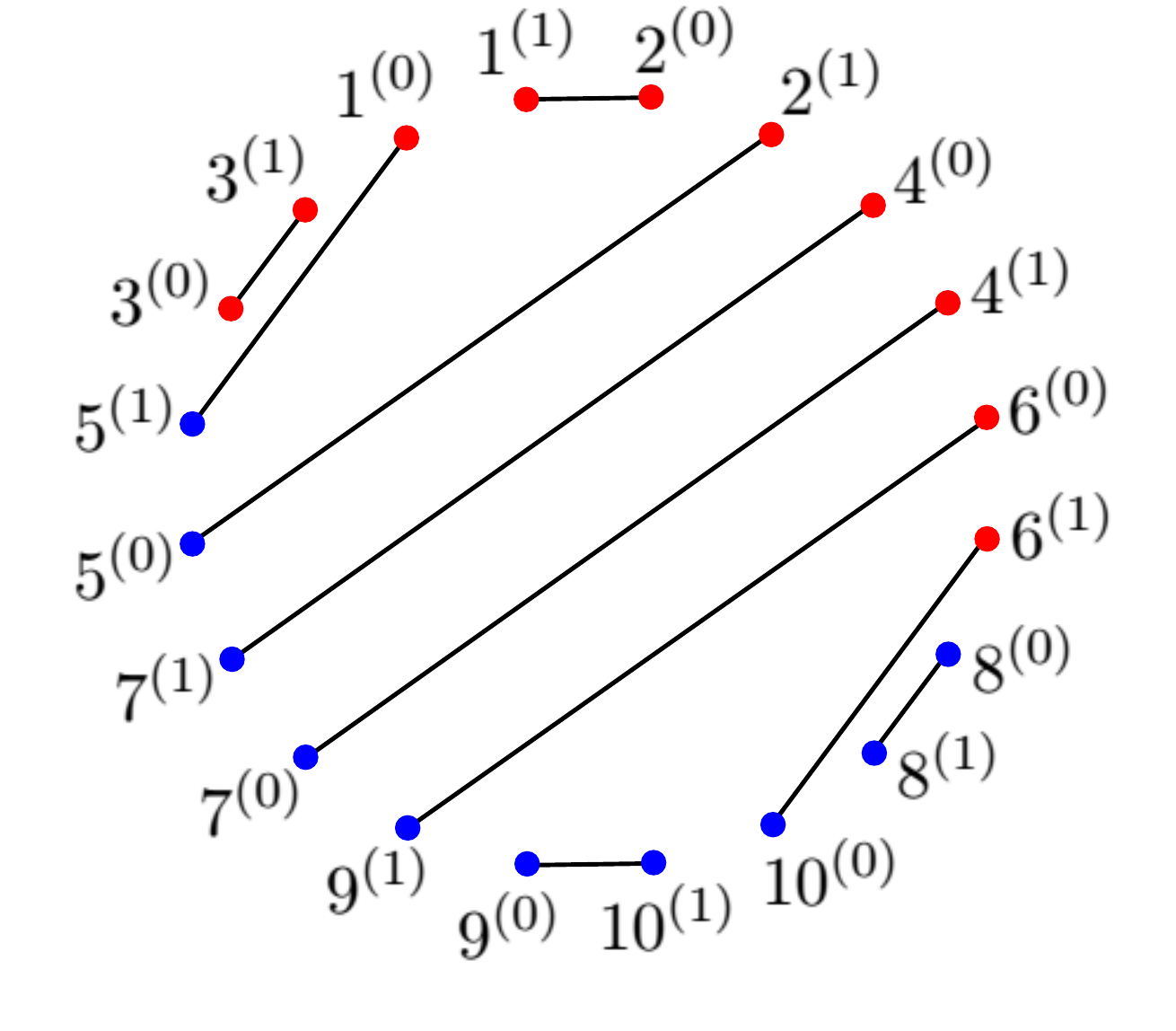}\hspace{1.28cm} \includegraphics[height=5.5cm]{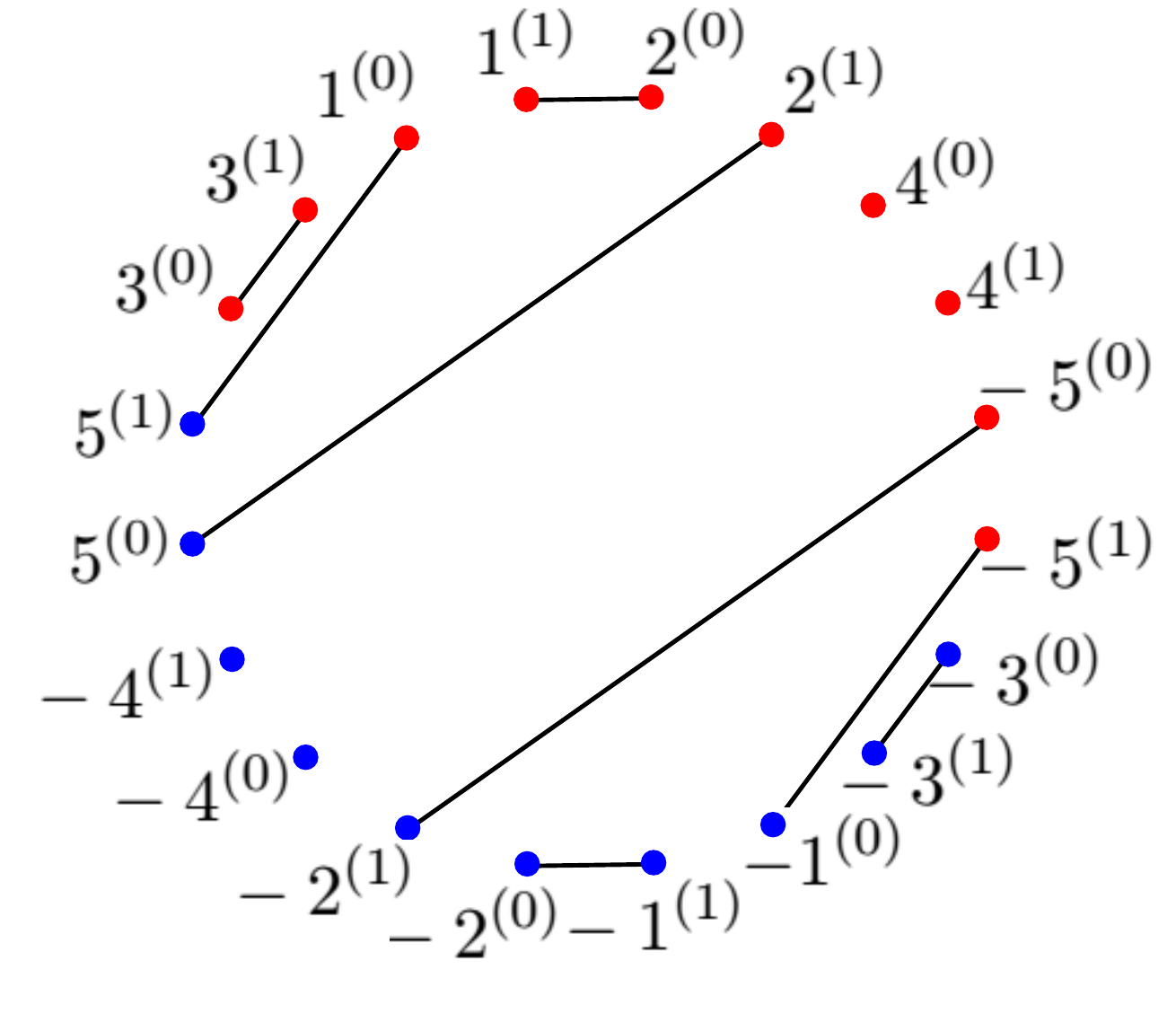}
\end{center}
\caption{The diagrams $\varphi_{A_9}(\widehat B)$ (left) and $\varphi_{D_6}(B)$ (right) from \cref{exam:final}.}\label{fig:examfig4}
\end{figure}

\begin{example}\label{exam:final}
Let $A\in\mathcal A(\Phi^+(D_6))$ and $\widehat A\in\mathcal A(\Phi^+(A_9))$ be the antichains from \cref{exam:astconstruction1}, as depicted in \cref{fig:examfig1}. Then $\row\cdot \Rvac(A)$ and $\row\cdot \Rvac(\widehat A)$ are the antichains shown in blue in \cref{fig:finalexam}. It is straightforward to check that $\row\cdot \Rvac(\widehat A)$ is the same as $\widehat{\row\cdot\Rvac(A)}$ in this case, as predicted by \cref{lem:HatCommutes}. Let $B=\row\cdot\Rvac(A)$. To help illustrate part of the proof of \cref{lem:HatCommutes}, we have drawn the diagrams $\varphi_{A_9}(\widehat B)$ and $\varphi_{D_6}(B)$ in \cref{fig:examfig4}. The diagrams $\varphi_{A_9}(\widehat A)$ and~$\varphi_{D_6}(A)$ are shown in \cref{fig:examfig3}. Notice that $\varphi_{A_9}(\widehat B)$ and $\varphi_{D_6}(B)$ are obtained by reflecting the edges in $\varphi_{A_9}(\widehat A)$ and $\varphi_{D_6}(A)$, respectively, through the line $M$.
\end{example}

Recall our notation $\alpha_1,\ldots,\alpha_n$ for the simple roots of $\Phi^+(D_n)$. We will now combine~\cref{lem:HatCommutes} with the following result of~\cite{armstrong2013uniform} to say that, for most $A$ with $\delta(A)\neq A$, we have $\widehat{\Rvac(A)} = \Rvac(\widehat{A})$.

\begin{lemma} \label{lem:HatRowCommutes}
Let $A\in\mathcal A(\Phi^+(D_n))$ be such that $\delta(A)\neq A$ and $\{\alpha_{n-1},\alpha_n\}\cap A = \varnothing$. Then 
\[\row(\widehat A)=\widehat{\row(A)}.\]
\end{lemma}
\begin{proof}
This is basically Lemma 3.16 of~\cite{armstrong2013uniform}. Specifically, that lemma says that if $A\in\mathcal A(\Phi^+(D_n))$ is such that $\delta(A)\neq A$ and $\widehat{A}\cap \mathcal{Q} \neq \varnothing$, then $\row(\widehat A)=\widehat{\row(A)}$. Hence, we must show that $\widehat{A}\cap \mathcal{Q} \neq \varnothing$. Recall the definition of $\mathcal{S}_{D_n}$ from \cref{subsec:RootPosets}. Because $\delta(A)\neq A$ and $\{\alpha_{n-1},\alpha_n\}\cap A=\varnothing$, there must be some $\alpha \in A\cap\mathcal{S}_{D_n}$ that is not in $\{\alpha_{n-1},\alpha_n\}$. Then $\iota(\gamma(\alpha))$ consists of two elements belonging to $\mathcal{Q}$, so $\#(\iota(\gamma(A))\cap\mathcal Q)\geq 2$. It is then immediate from the definition of $\widehat{A}$ that $\#(\widehat{A}\cap \mathcal{Q}) = \#(\iota(\gamma(A))\cap \mathcal{Q})-1\geq 1$. 
\end{proof}

\begin{cor} \label{cor:HatCommutes}
If $A\in\mathcal A(\Phi^+(D_n))$ is such that $\delta(A)\neq A$ and $\{\alpha_{n-1},\alpha_n\}\cap A = \varnothing$, then 
\[\Rvac(\widehat A)=\widehat{\Rvac(A)}.\]
\end{cor}
\begin{proof}
From~\cref{lem:HatRowCommutes} we have
\[ \row(\widehat{A}) = \widehat{\row(A)}.\]
As mentioned above, rowmotion commutes with poset automorphisms, so it is still the case that $\delta(\row(A)) \neq \row(A)$. Then we can apply~\cref{lem:HatCommutes} to say that
\[\row\cdot \Rvac\cdot \row(\widehat A)= \row\cdot \Rvac (\widehat{\row(A)}) =\widehat{\row\cdot \Rvac\cdot \row(A)}.\]
The result follows from~\cref{prop:row_basics}, which says $\row\cdot \Rvac\cdot \row=\Rvac$.
\end{proof}

\begin{figure}[ht]
\begin{center}
\includegraphics[height=3.61cm]{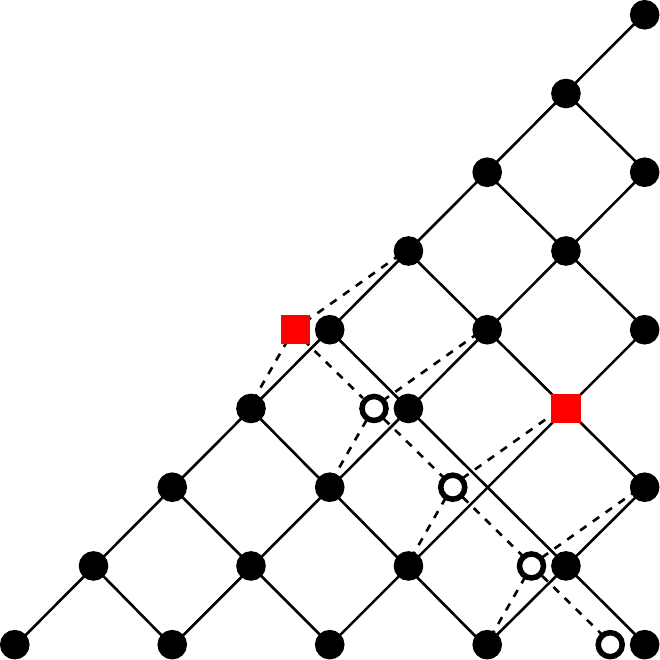}\hspace{1.7cm} \includegraphics[height=3.61cm]{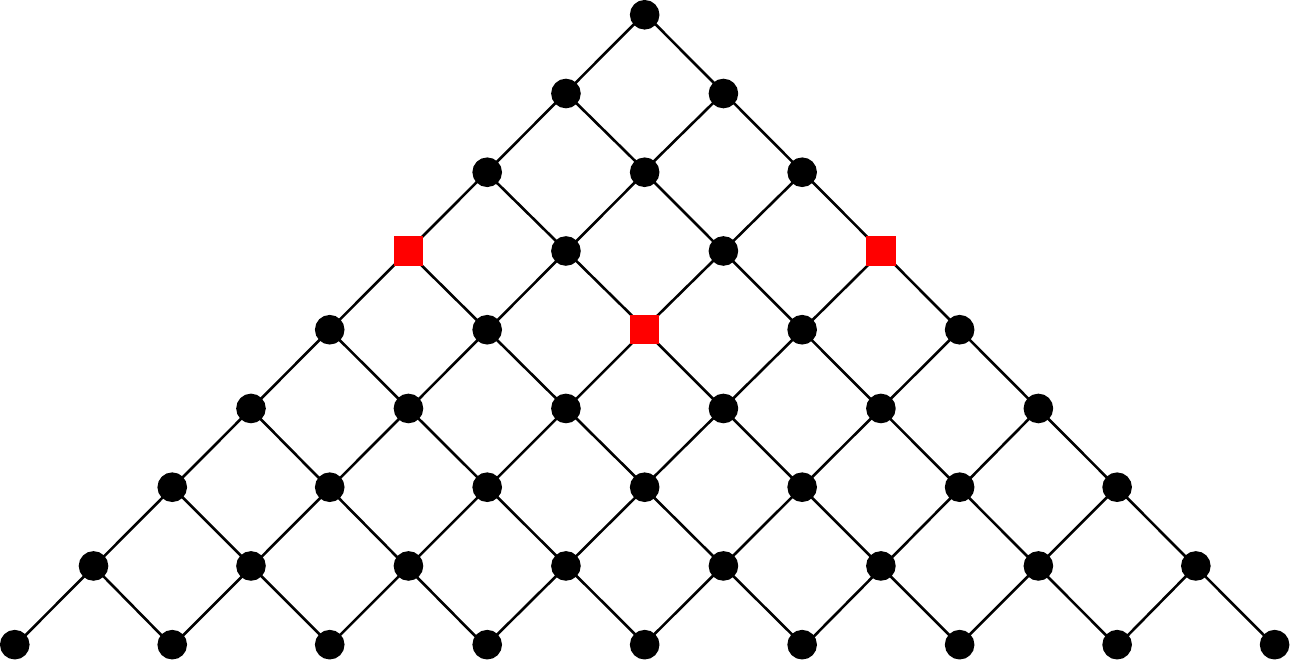} \caption{The antichains $\Rvac(A)$ (left) and $\Rvac(\widehat A)$ (right), where $A$ and $\widehat A$ are the antichains from \cref{exam:astconstruction1}.}\label{fig:finalfinalexam}
\end{center}
\end{figure}

\begin{example}
Let $A\in\mathcal A(\Phi^+(D_6))$ and $\widehat A\in\mathcal A(\Phi^+(A_9))$ be the antichains from \cref{exam:astconstruction1}, as depicted in \cref{fig:examfig1}. Then $\Rvac(A)$ and $\Rvac(\widehat A)$ are the antichains shown in red in \cref{fig:finalfinalexam}. It is straightforward to check that $\Rvac(\widehat A)$ is the same as $\widehat{\Rvac(A)}$ in this case, as predicted by \cref{cor:HatCommutes}.
\end{example}

\begin{remark}
The assumption $\{\alpha_{n-1},\alpha_n\}\cap A = \varnothing$ in~\cref{cor:HatCommutes} is required. For example, if $A=\{\alpha_n\}$, then $\widehat{A}=\varnothing$, so $\Rvac(\widehat{A})$ consists of all minimal elements of $\Phi^+({A_{2n-3}})$; but $\widehat{\Rvac(A)} = \widehat{\{\alpha_1,\ldots,\alpha_{n-1}\}}$ consists of all but one of the minimal elements of $\Phi^+({A_{2n-3}})$ (it is missing the ``middle'' simple root $[n-1,n]$).
\end{remark}

\section{Proof of Panyushev's Conjecture in Type D}\label{Sec:MainProof}

Recall from~\cref{subsec:pan_conj} that our goal is to prove the following theorem. 

\begin{thm}\label{thm:PanyushevD}
For every $A\in\mathcal A(\Phi^+(D_n))$, we have $\#A+\#\Rvac(A)=n$. 
\end{thm}

In proving~\cref{thm:PanyushevD}, we will make heavy use of \cref{thm:main_intro} for Types~A and~C. As a consequence, we will need to consider the action of rowvacuation on different posets. As before, we will use subscripts as in the notation $\Rvac_P$ when referring to an operator on $\mathcal A(P)$. We will primarily be concerned with the poset~$\Phi^+(D_n)$, so operators without subscripts will be assumed to act on~$\mathcal{A}(\Phi^+(D_n))$. 

Before we start with the proof of~\cref{thm:PanyushevD}, we need two simple propositions about rowvacuation of $\eta$-symmetric antichains in Type~A. Recall the definitions of the special subsets $\mathcal{L}$, $\mathcal{S}$ of the root posets of classical types from~\cref{subsec:RootPosets}.

\begin{prop} \label{prop:typec_dist}
Let $B\in\mathcal{A}(\Phi^+(A_{2n-1}))$ be such that $\eta(B)=B$. Then exactly one of $B$ and $\Rvac_{\Phi^+(A_{2n-1})}(B)$ contains an element of $\mathcal{L}_{A_{2n-1}}$.
\end{prop}

\begin{prop} \label{prop:typeb_dist}
Let $B\in\mathcal{A}(\Phi^+(A_{2n-1}))$ be such that $\eta(B)=B$. Then exactly one of $B$ and $\Rvac_{\Phi^+(A_{2n-1})}(B)$ contains an element of $\mathcal{S}_{A_{2n-1}}$.
\end{prop}

\begin{proof}[Proofs of~\cref{prop:typec_dist,prop:typeb_dist}]
Note that~\cref{prop:typec_dist} can be seen as the Type~C case of~\cref{conj:pan}(ii), i.e., the fact that $\iota^{-1}(B)\cup\Rvac_{\Phi^+(C_n)}(\iota^{-1}(B))$ has the right distribution of long and short roots. Similarly,~\cref{prop:typeb_dist} can be seen as the Type~B case of~\cref{conj:pan}(ii), i.e., the fact that $\Rvac$ on the Type~B root poset has the right distribution of long and short roots. So they can be thought of as consequences of Panyushev's work in~\cite[\S5]{panyushev2004adnilpotent}. However, they are also both immediate from the formula for Type~A rowvacuation in~\cref{thm:lk}.
\end{proof}

Our proof of~\cref{thm:PanyushevD} will consist of several cases. We first dispense with the case of \cref{thm:PanyushevD} in which $\delta(A)=A$. 

\begin{lemma}\label{lem:delta(A)=A}
If $A\in\mathcal A(\Phi^+(D_n))$ is such that $\delta(A)=A$, then $\#A+\#\Rvac(A)=n$. 
\end{lemma}

\begin{proof}
Because $\mathcal{S}_{C_{n-1}}$ is a chain, each antichain of $\Phi^+(C_{n-1})$ contains at most one element of $\mathcal{S}_{C_{n-1}}$. Given an antichain~$B$ of $\Phi^+(C_{n-1})$, let $\varepsilon(B)$ be $1$ if $B$ contains an element of $\mathcal{S}_{C_{n-1}}$ and~$0$ otherwise. We have $\#\gamma^{-1}(B)=\#B+\varepsilon(B)$. 

Because rowvacuation commutes with the poset automorphism $\delta$, we have that $\delta(\Rvac(A))=\Rvac(A)$. One can easily check that~$\gamma(A)$ is an antichain of $\Phi^+(C_{n-1})$ for which $\Rvac_{\Phi^+(C_{n-1})}(\gamma(A))=\gamma(\Rvac(A))$. So, denoting $\ast\coloneqq \#A+\#\Rvac(A)$, we have
\begin{align*}
\ast &=\#\gamma^{-1}(\gamma(A))+\#\gamma^{-1}(\gamma(\Rvac(A)))\\
&=\#\gamma^{-1}(\gamma(A))+\#\gamma^{-1}(\Rvac_{\Phi^+(C_{n-1})}(\gamma(A)))\\
&=\#\gamma(A)+\varepsilon(\gamma(A))+\#\Rvac_{\Phi^+(C_{n-1})}(\gamma(A))+\varepsilon(\Rvac_{\Phi^+(C_{n-1})}(\gamma(A))).
\end{align*} 

We have seen in~\cref{subsec:pan_conj} that rowvacuation serves as Panyushev's~$\mathfrak{P}$ in Type~C, which means that $\#\gamma(A)+\#\Rvac_{\Phi^+(C_{n-1})}(\gamma(A))=n-1$. Hence, in order to prove $\#A+\#\Rvac(A)=n$, we need to show that 
\[\varepsilon(\gamma(A))+\varepsilon(\Rvac_{\Phi^+(C_{n-1})}(\gamma(A)))=1.\] 
In other words, we need to show that $\gamma(A)$ contains an element of $\mathcal{S}_{C_{n-1}}$ if and only if $\Rvac_{\Phi^+(C_{n-1})}(\gamma(A))$ does not contain an element of $\mathcal{S}_{C_{n-1}}$. It suffices to show that $\iota(\gamma(A))$ contains an element of $\mathcal{S}_{A_{2n-3}}$ if and only if $\iota(\Rvac_{\Phi^+(C_{n-1})}(\gamma(A)))$ does not contain an element of $\mathcal{S}_{A_{2n-3}}$. Since 
\[\iota(\Rvac_{\Phi^+(C_{n-1})}(\gamma(A)))=\Rvac_{\Phi^+(A_{2n-3})}(\iota(\gamma(A))),\] 
this follows from~\cref{prop:typeb_dist}.
\end{proof}

Now we consider the case where $\delta(A)\neq A$. Actually, we have two cases here, depending on whether $A$ contains an element of $\{\alpha_{n-1},\alpha_n\}$ or not.

\begin{lemma}\label{lem:delta(A)neqA1}
If $A\in\mathcal A(\Phi^+(D_n))$ is such that $\delta(A)\neq A$ and $\{\alpha_{n-1},\alpha_n\}\cap A =\varnothing$, then $\#A+\#\Rvac(A)=n$. 
\end{lemma}
\begin{proof}
Recall the antichain~$\widehat{A}\in \mathcal{A}(\Phi^+(A_{2n-3}))$ defined in terms of $A$ in~\cref{Sec:ASTBijectionD}. From the definition of~$\widehat{A}$, it is immediate that
\[ \#\widehat{A} = \begin{cases} 2\cdot\#A - 1& \textrm{if $A\cap \mathcal{L}_{D_n}=\varnothing$}; \\
2\cdot\#A - 2 &\textrm{if $A\cap \mathcal{L}_{D_n}\neq \varnothing$}.\end{cases}\]
Given an antichain~$B$ of $\Phi^+(A_{2n-3})$, let $\varepsilon(B)$ be $1$ if $B$ contains an element of $\mathcal{L}_{A_{2n-3}}$ and~$0$ otherwise. It is straightforward to check from the definition of~$\widehat{A}$ that $A\cap \mathcal{L}_{D_n}=\varnothing$ if and only if $\widehat{A}\cap\mathcal{L}_{A_{2n-3}}\neq \varnothing$. Hence, $\#\widehat{A} = 2\cdot\#A - 2 +\varepsilon(\widehat{A})$. Similarly, we have $\#\widehat{\Rvac(A)} = 2\cdot\#\Rvac(A) - 2 +\varepsilon(\widehat{\Rvac(A)})$.

From~\cref{cor:HatCommutes}, we get
\begin{align*}
\#\widehat{A}\!+\!\#\Rvac_{\Phi^+({A_{2n-3}})}(\widehat{A}) &= \#\widehat{A}+\#\widehat{\Rvac(A)} \\
&= 2(\#A\!+\!\#\Rvac(A)) - 4 +\varepsilon(\widehat{A})+\varepsilon(\widehat{\Rvac(A)})\\
&= 2(\#A\!+\!\#\Rvac(A)) - 4 +\varepsilon(\widehat{A})+\varepsilon(\Rvac_{\Phi^+({A_{2n-3}})}(\widehat{A})).
\end{align*}
But~\cref{prop:typec_dist} tells us that $\varepsilon(\widehat{A})+\varepsilon(\Rvac_{\Phi^+({A_{2n-3}})}(\widehat{A}))=1$, so
\[\#\widehat{A}+\#\Rvac_{\Phi^+({A_{2n-3}})}(\widehat{A})=2(\#A+\#\Rvac(A)) - 3.\]
As we saw in~\cref{subsec:pan_conj}, rowvacuation serves as Panyushev's~$\mathfrak{P}$ in Type~A, so $\#\widehat{A}+\#\Rvac_{\Phi^+({A_{2n-3}})}(\widehat{A})=2n-3$. Hence, $\#A+\#\Rvac(A)=n$.
\end{proof}

\begin{lemma}\label{lem:delta(A)neqA2}
If $A\in\mathcal A(\Phi^+(D_n))$ is such that $\delta(A)\neq A$ and $\{\alpha_{n-1},\alpha_n\}\cap A \neq\varnothing$, then $\#A+\#\Rvac(A)=n$. 
\end{lemma}
\begin{proof}
Assume without loss of generality that $\alpha_n\in A$. Let $\Phi'$ be the maximal parabolic sub-root system of $\Phi^+(D_n)$ obtained by removing $\alpha_n$ from the system of simple roots. Clearly, $\Phi'\simeq A_{n-1}$. From the fact that rowvacuation respects parabolic induction -- i.e., from~\cref{conj:pan}(iv) or, more precisely, \cref{prop:rvac_parabolic} -- we know that $\Rvac(A)=\Rvac_{(\Phi')^+}(A\setminus \{\alpha_n\})$. But since we know that rowvacuation is $\mathfrak{P}$ in Type~A, we know that $\#\Rvac_{(\Phi')^+}(A\setminus \{\alpha_n\})=(n-1)-(\#A-1)=n-\#A$. Hence, $\#A+\#\Rvac(A)=n$. (Note we did not need the hypothesis $\delta(A)\neq A$.)
\end{proof}

Altogether, \cref{lem:delta(A)=A,lem:delta(A)neqA1,lem:delta(A)neqA2} imply~\cref{thm:PanyushevD}. Hence, we have completed the proof of~\cref{thm:main_intro}.

\bibliography{narayana_rowvacuation}{}
\bibliographystyle{abbrv}

\end{document}